\documentclass[12pt,fleqn]{article}

\usepackage{amsmath,amssymb,amsthm}
\usepackage{geometry}
\usepackage[pdfpagemode=UseNone,pdfstartview=FitH]{hyperref}

\newcommand{\abs}[1]{\left|#1\right|}
\newcommand{\bdry}[1]{\partial #1}
\newcommand{\bigset}[1]{\big\{#1\big\}}
\newcommand{\B}{{\cal B}}
\newcommand{\D}{{\cal D}}
\newcommand{\closure}[1]{\overline{#1}}
\newcommand{\dint}{\ds{\int}}
\newcommand{\dist}[2]{\text{dist}\, (#1,#2)}
\newcommand{\ds}[1]{\displaystyle #1}
\newcommand{\eps}{\varepsilon}
\newcommand{\goodchi}{\protect\raisebox{2pt}{$\chi$}}
\newcommand{\half}{\frac{1}{2}}
\newcommand{\interior}[1]{#1^\circ}
\newcommand{\ipa}[3][]{\left<#2,#3\right>_{#1}}
\renewcommand{\L}{{\mathcal L}}
\newcommand{\M}{{\cal M}}
\newcommand{\norm}[2][]{\left\|#2\right\|_{#1}}
\renewcommand{\O}{\text{O}}
\renewcommand{\o}{\text{o}}
\newcommand{\PS}[1]{$(\text{PS})_{#1}$}
\newcommand{\pnorm}[2][]{\if #1'' \left|#2\right|_p \else \left|#2\right|_{#1} \fi}
\newcommand{\R}{\mathbb R}
\newcommand{\restr}[2]{\left.#1\right|_{#2}}
\newcommand{\seq}[1]{\left(#1\right)}
\newcommand{\set}[1]{\left\{#1\right\}}
\newcommand{\vol}[1]{\if #1'' \L \else \L(#1) \fi}
\newcommand{\W}{{\cal W}}
\newcommand{\wto}{\rightharpoonup}

\def\ocirc#1{\ifmmode\setbox0=\hbox{$#1$}\dimen0=\ht0
\advance\dimen0 by1pt\rlap{\hbox to\wd0{\hss\raise\dimen0
\hbox{\hskip.2em$\scriptscriptstyle\circ$}\hss}}#1\else
{\accent"17 #1}\fi}

\def\Xint#1{\mathchoice
{\XXint\displaystyle\textstyle{#1}}
{\XXint\textstyle\scriptstyle{#1}}
{\XXint\scriptstyle\scriptscriptstyle{#1}}
{\XXint\scriptscriptstyle\scriptscriptstyle{#1}}
\!\int}
\def\XXint#1#2#3{{\setbox0=\hbox{$#1{#2#3}{\int}$}
\vcenter{\hbox{$#2#3$}}\kern-.5\wd0}}
\def\dashint{\Xint-}

\DeclareMathOperator{\divg}{div}

\newenvironment{enumarab}{\begin{enumerate}

}{\end{enumerate}}

\newenvironment{enumroman}{\begin{enumerate}

}{\end{enumerate}}

\newtheorem{lemma}{Lemma}[section]
\newtheorem{proposition}[lemma]{Proposition}
\newtheorem{theorem}[lemma]{Theorem}

\theoremstyle{definition}
\newtheorem{definition}[lemma]{Definition}
\newtheorem*{notation}{Notation}

\numberwithin{equation}{section}

\title{\bf On nonminimizing solutions of elliptic free boundary problems}
\author{\bf Kanishka Perera\\
Department of Mathematical Sciences\\
Florida Institute of Technology\\
Melbourne, FL 32901, USA\\
\em kperera@fit.edu}
\date{}

\begin{document}

\maketitle

\begin{abstract}
We present a variational framework for studying the existence and regularity of solutions to elliptic free boundary problems that do not necessarily minimize energy. As applications, we obtain mountain pass solutions of critical and subcritical superlinear free boundary problems, and establish full regularity of the free boundary in dimension $N = 2$ and partial regularity in higher dimensions.
\end{abstract}

\section{Introduction}

Existence and regularity of minimizers in elliptic free boundary problems have been studied extensively in the literature (see, e.g., \cite{MR618549, MR732100, MR1044809, MR2145284, MR861482, MR990856, MR1029856, MR973745, MR1906591, MR2082392, MR2572253, MR1009785, MR0440187, MR1664689, MR2281453, MR1620644, MR1759450} and the references therein). The purpose of this paper is to present a variational framework for studying the existence and regularity of solutions that do not necessarily minimize energy.

Let $\Omega$ be a bounded domain in $\R^N,\, N \ge 2$ with $C^{2,\alpha}$-boundary $\bdry{\Omega}$. We consider the problem
\begin{equation} \label{1001}
\left\{\begin{aligned}
- \Delta u & = g(x,(u - 1)_+) && \text{in } \Omega \setminus F(u)\\[10pt]
|\nabla u^+|^2 - |\nabla u^-|^2 & = 2 && \text{on } F(u)\\[10pt]
u & = 0 && \text{on } \bdry{\Omega},
\end{aligned}\right.
\end{equation}
where $F(u) = \bdry{\set{u > 1}}$ is the free boundary of $u$, $(u - 1)_+ = \max\, (u - 1,0)$ is the positive part of $u - 1$, $\nabla u^\pm$ are the limits of $\nabla u$ from the sets $\set{u > 1}$ and $\interior{\set{u \le 1}}$, respectively, and $g$ is a locally H\"{o}lder continuous function on $\Omega \times [0,\infty)$ satisfying
\begin{enumerate}
\item[$(g_1)$] $g(x,0) = 0$ for all $x \in \Omega$,
\item[$(g_2)$] for some $a_1, a_2 > 0$,
    \[
    |g(x,s)| \le \begin{cases}
    a_1\, e^{a_2 s^2} & \text{if } N = 2\\[7.5pt]
    a_1 + a_2\, s^{2^\ast - 1} & \text{if } N \ge 3
    \end{cases}
    \]
    for all $(x,s) \in \Omega \times [0,\infty)$, where $2^\ast = 2N/(N - 2)$ is the critical Sobolev exponent when $N \ge 3$.
\end{enumerate}
The right-hand side of the equation $- \Delta u = g(x,(u - 1)_+)$ is zero on the free boundary by $(g_1)$. The growth condition $(g_2)$ ensures that the associated variational functional
\[
J(u) = \int_\Omega \left[\half\, |\nabla u|^2 + \goodchi_{\set{u > 1}}(x) - G(x,(u - 1)_+)\right] dx,
\]
where $\goodchi_{\set{u > 1}}$ is the characteristic function of the set $\set{u > 1}$ and $G(x,s) = \int_0^s g(x,t)\, dt$ is the primitive of $g$, is defined on the Sobolev space $H^1_0(\Omega)$.

Since the functional $J$ is not differentiable, we approximate it by $C^1$-functionals as follows. Let $\beta : \R \to [0,2]$ be a smooth function such that $\beta(s) = 0$ for $s \le 0$, $\beta(s) > 0$ for $0 < s < 1$, $\beta(s) = 0$ for $s \ge 1$, and $\int_0^1 \beta(t)\, dt = 1$. For $\eps > 0$, let
\[
J_\eps(u) = \int_\Omega \left[\half\, |\nabla u|^2 + \B\left(\frac{u - 1}{\eps}\right) - G(x,(u - 1)_+)\right] dx, \quad u \in H^1_0(\Omega),
\]
where $\B(s) = \int_0^s \beta(t)\, dt$. First we prove a general convergence result for a sequence $\seq{u_j}$ of critical points of $J_{\eps_j}$ that is bounded in $H^1_0(\Omega) \cap L^\infty(\Omega)$, where $\eps_j \searrow 0$ (see Theorem \ref{Theorem 1}). The singular limit in this theorem is a Lipschitz continuous function $u \in H^1_0(\Omega) \cap C^2(\closure{\Omega} \setminus F(u))$ that satisfies the inequality $- \Delta u \le g(x,(u - 1)_+)$ in the distributional sense in $\Omega$ and the equation $- \Delta u = g(x,(u - 1)_+)$ in the classical sense in $\Omega \setminus F(u)$. If, in addition, $u$ is nondegenerate (see Definition \ref{Definition 2}), then by the results of Lederman and Wolanski \cite{MR2281453}, $u$ also satisfies the free boundary condition in the weak viscosity sense (see Definition \ref{Definition 4}). Next we show that the case where $u > 1$ on both sides of the free boundary can be ruled out and we can obtain a stronger form of viscosity solution if $u$ also has the positive density property for $\set{u > 1}$ and $\set{u \le 1}$ (see Definition \ref{Definition 3} and Proposition \ref{Proposition 1}). We also show that if $J_{\eps_j}(u_j) \to J(u)$, then $u$ satisfies the free boundary condition in the variational sense (see Definition \ref{Definition 1} and Proposition \ref{Proposition 2}).

Our main regularity result establishes full regularity of the free boundary in dimension $N = 2$ and partial regularity in higher dimensions for Lipschitz continuous solutions $u \in H^1_0(\Omega) \cap C^2(\closure{\Omega} \setminus F(u))$ that are nondegenerate, have the positive density property for $\set{u > 1}$ and $\set{u \le 1}$, and satisfy the free boundary condition in the viscosity sense and in the variational sense (see Theorem \ref{Theorem 4}). To apply this result to the singular limit $u$ in Theorem \ref{Theorem 1}, we have to show that $u$ is nondegenerate, has the positive density property for $\set{u > 1}$ and $\set{u \le 1}$, and $J_{\eps_j}(u_j) \to J(u)$. Then $u$ satisfies the free boundary condition in the viscosity sense by Proposition \ref{Proposition 1} and in the variational sense by Proposition \ref{Proposition 2}, so the conclusions of Theorem \ref{Theorem 4} hold for $u$. In particular, the free boundary of $u$ has finite $(N - 1)$-dimensional Hausdorff measure and is smooth except on a closed set of Hausdorff dimension at most $N - 3$. We carry out this program for a class of superlinear free boundary problems next.

We assume that the nonlinearity $g$ satisfies, in addition to $(g_1)$ and $(g_2)$,
\begin{enumerate}
\item[$(g_3)$] $g(x,s) > 0$ for all $x \in \Omega$ and $s > 0$,
\item[$(g_4)$] there exists $\mu > 2$ such that the mappings
    \[
    s \mapsto \frac{g(x,s)}{s^{\mu - 1}}, \qquad s \mapsto \frac{1}{\mu}\, s g(x,s) - G(x,s)
    \]
    are nondecreasing for all $x \in \Omega$ and $s > 0$.
\end{enumerate}
For example, a sum of powers
\[
g(x,s) = \sum_{i=1}^n s^{p_i - 1},
\]
where $p_i > 2$ if $N = 2$ and $2 < p_i \le 2^\ast$ if $N \ge 3$, satisfies $(g_1)$--$(g_4)$ with $\mu = \min p_i$. The conditions $(g_3)$ and $(g_4)$ imply that the functional $J$ has the mountain pass geometry (see Lemma \ref{Lemma 200}). Let
\[
\Gamma = \set{\gamma \in C([0,1],H^1_0(\Omega)) : \gamma(0) = 0,\, J(\gamma(1)) < 0}
\]
be the class of paths joining the origin to the set $\set{u \in H^1_0(\Omega) : J(u) < 0}$, and let
\[
c := \inf_{\gamma \in \Gamma}\, \max_{u \in \gamma([0,1])}\, J(u) > 0
\]
be the mountain pass level. We prove the existence and regularity of a nonminimizing solution of mountain pass type at this level in both critical and subcritical cases.

Let $\W = \bigset{u \in H^1_0(\Omega) : u^+ \ne 0}$, where $u^+ = (u - 1)_+$. All nontrivial solutions of problem \eqref{1001} lie on the Nehari manifold
\[
\M = \set{u \in \W : \int_\Omega |\nabla u^+|^2\, dx = \int_\Omega u^+ g(x,u^+)\, dx}.
\]
We will take a sequence $\eps_j \searrow 0$ and show that each approximating functional $J_{\eps_j}$ has a critical point $u_j$ of mountain pass type. Then we will apply Theorem \ref{Theorem 1} and show that the singular limit $u$ is in $\M$ and satisfies
\[
J(u) = c = \inf_{v \in \M}\, J(v),
\]
so $u$ is a minimizer of $\restr{J}{\M}$. We will then use this fact to show that $u$ is a mountain pass point of $J$ (see Proposition \ref{Proposition 4}) and that $u$ is nondegenerate and has the positive density property for $\set{u > 1}$ and $\set{u \le 1}$ (see Proposition \ref{Proposition 3}). This will alow us to apply Theorem \ref{Theorem 4} and establish full regularity of the free boundary in dimension $N = 2$ and partial regularity in higher dimensions for our mountain pass solution.

The subcritical pure power case $g(x,s) = s^{p-1}$, where $p > 2$ if $N = 2$ and $2 < p < 2^\ast$ if $N \ge 3$, of problem \eqref{1001} was considered in Jerison and Perera \cite{MR3790500}. However, proving the nondegeneracy and the positive density property of minimizers of $\restr{J}{\M}$ in the general case considered here is substantially more difficult. The Nehari manifold in the pure power case has the explicit description
\[
\M = \set{u^- + t_u u^+ : u \in \W},
\]
where $u^- = u - u^+$ and
\[
t_u = \left[\frac{\dint_{\set{u > 1}} |\nabla u|^2\, dx}{\dint_{\set{u > 1}} (u - 1)^p\, dx}\right]^{1/(p-2)}.
\]
The proofs in \cite{MR3790500} make extensive use of this description. No such description of the Nehari manifold is available in the general case, even in the next simplest case of a sum of two different powers. This makes the analysis of the Nehari manifold needed to obtain nondegeneracy estimates much more difficult in the general case. This analysis is carried out in Section \ref{Nondegeneracy}.

Moreover, in the critical case $p = 2^\ast$ in dimensions $N \ge 3$, which was not considered in \cite{MR3790500}, the approximating functionals $J_{\eps_j}$ do not satisfy the \PS{c} condition for all $c \in \R$ due to the noncompactness of the critical Sobolev embedding $H^1_0(\Omega) \hookrightarrow L^{2^\ast}(\Omega)$. We consider the model nonlinearity $g(x,s) = \kappa s^{2^\ast - 1} + \lambda s^{\mu - 1}$, where $\kappa, \lambda > 0$ are parameters and $2 < \mu < 2^\ast$. To overcome the difficulties arising from the lack of compactness, first we show that if
\[
0 < c < \frac{1}{N}\, \frac{S^{N/2}}{\kappa^{N/2 - 1}},
\]
where $S$ is the best constant for the Sobolev embedding, then every \PS{c} sequence of $J_{\eps_j}$ has a subsequence that converges weakly to a nontrivial critical point $u$ of $J_{\eps_j}$ satisfying $J_{\eps_j}(u) \le c$ (see Lemma \ref{Lemma 301}). Then we show that the mountain pass level of each $J_{\eps_j}$ is below this compactness threshold when $\kappa$ is sufficiently small. This allows us to establish the existence and partial regularity of a mountain pass solution for all sufficiently small $\kappa > 0$ and all $\lambda > 0$. The limiting case $\mu = 2$ of this problem was considered in Yang and Perera \cite{MR3872792}, and a nondegenerate mountain pass solution that satisfies the free boundary condition in the viscosity sense was obtained for all sufficiently small $\kappa > 0$ and $\lambda > \lambda_1$, the first Dirichlet eigenvalue of $- \Delta$ in $\Omega$. However, the question of regularity of the free boundary was not considered in \cite{MR3872792}.

\section{Statement of results}

Let $\Omega$ be a bounded domain in $\R^N,\, N \ge 2$ with $C^{2,\alpha}$-boundary $\bdry{\Omega}$. We consider the problem
\begin{equation} \label{1}
\left\{\begin{aligned}
- \Delta u & = g(x,(u - 1)_+) && \text{in } \Omega \setminus F(u)\\[10pt]
|\nabla u^+|^2 - |\nabla u^-|^2 & = 2 && \text{on } F(u)\\[10pt]
u & = 0 && \text{on } \bdry{\Omega},
\end{aligned}\right.
\end{equation}
where
\[
F(u) = \bdry{\set{u > 1}}
\]
is the free boundary of $u$, $(u - 1)_+ = \max\, (u - 1,0)$ is the positive part of $u - 1$, $\nabla u^\pm$ are the limits of $\nabla u$ from the sets $\set{u > 1}$ and $\interior{\set{u \le 1}}$, respectively, and $g$ is a locally H\"{o}lder continuous function on $\Omega \times [0,\infty)$ satisfying
\begin{enumerate}
\item[$(g_1)$] $g(x,0) = 0$ for all $x \in \Omega$,
\item[$(g_2)$] for some $a_1, a_2 > 0$,
    \[
    |g(x,s)| \le \begin{cases}
    a_1\, e^{a_2 s^2} & \text{if } N = 2\\[7.5pt]
    a_1 + a_2\, s^{2^\ast - 1} & \text{if } N \ge 3
    \end{cases}
    \]
    for all $(x,s) \in \Omega \times [0,\infty)$, where $2^\ast = 2N/(N - 2)$ is the critical Sobolev exponent when $N \ge 3$.
\end{enumerate}

The right-hand side of the equation $- \Delta u = g(x,(u - 1)_+)$ is zero on the free boundary by $(g_1)$. The growth condition $(g_2)$ ensures that the associated variational functional
\[
J(u) = \int_\Omega \left[\half\, |\nabla u|^2 + \goodchi_{\set{u > 1}}(x) - G(x,(u - 1)_+)\right] dx,
\]
where $\goodchi_{\set{u > 1}}$ is the characteristic function of the set $\set{u > 1}$ and
\[
G(x,s) = \int_0^s g(x,t)\, dt, \quad s \ge 0
\]
is the primitive of $g$, is defined on the Sobolev space $H^1_0(\Omega)$.

However, the functional $J$ is nondifferentiable, so we approximate it by \linebreak $C^1$-functionals as follows. Let $\beta : \R \to [0,2]$ be a smooth function such that $\beta(s) = 0$ for $s \le 0$, $\beta(s) > 0$ for $0 < s < 1$, $\beta(s) = 0$ for $s \ge 1$, and $\int_0^1 \beta(t)\, dt = 1$. Then set
\[
\B(s) = \int_0^s \beta(t)\, dt,
\]
and note that $\B : \R \to [0,1]$ is a smooth nondecreasing function such that $\B(s) = 0$ for $s \le 0$, $0 < \B(s) < 1$ for $0 < s < 1$, and $\B(s) = 1$ for $s \ge 1$. For $\eps > 0$, let
\[
J_\eps(u) = \int_\Omega \left[\half\, |\nabla u|^2 + \B\left(\frac{u - 1}{\eps}\right) - G(x,(u - 1)_+)\right] dx, \quad u \in H^1_0(\Omega),
\]
and note that the functional $J_\eps$ is of class $C^1$.

Critical points of $J_\eps$ coincide with weak solutions of the problem
\begin{equation} \label{13}
\left\{\begin{aligned}
- \Delta u & = - \frac{1}{\eps}\, \beta\left(\frac{u - 1}{\eps}\right) + g(x,(u - 1)_+) && \text{in } \Omega\\[10pt]
u & = 0 && \text{on } \bdry{\Omega}.
\end{aligned}\right.
\end{equation}
If $u$ is a weak solution of this problem, then it is also a classical $C^{2,\alpha}$-solution by elliptic regularity theory. If $u$ is not identically zero, then it is nontrivial in a stronger sense, namely, $u > 0$ in $\Omega$ and $u > 1$ in a nonempty open set. Indeed, if $u \le 1$ everywhere, then $u$ is harmonic in $\Omega$ by $(g_1)$ and hence vanishes identically since $u = 0$ on $\bdry{\Omega}$. Furthermore, in the set $\set{u < 1}$, $u$ is the harmonic function with boundary values $0$ on $\bdry{\Omega}$ and $1$ on $\bdry{\set{u \ge 1}}$, and hence strictly positive since $\Omega$ is connected.

Our main convergence result as $\eps \searrow 0$ is the following theorem.

\begin{theorem} \label{Theorem 1}
Assume $(g_1)$ and $(g_2)$. Let $\eps_j \searrow 0$ and let $u_j$ be a critical point of $J_{\eps_j}$. If $\seq{u_j}$ is bounded in $H^1_0(\Omega) \cap L^\infty(\Omega)$, then there exists a Lipschitz continuous function $u$ on $\closure{\Omega}$ such that $u \in H^1_0(\Omega) \cap C^2(\closure{\Omega} \setminus F(u))$ and, for a renamed subsequence,
\begin{enumroman}
\item \label{Theorem 1.i} $u_j \to u$ uniformly on $\closure{\Omega}$,
\item \label{Theorem 1.ii} $u_j \to u$ strongly in $H^1_0(\Omega)$,
\item \label{Theorem 1.iii} $J(u) \le \varliminf J_{\eps_j}(u_j) \le \varlimsup J_{\eps_j}(u_j) \le J(u) + \vol{{\set{u = 1}}}$, where $\vol{}$ denotes the Lebesgue measure in $\R^N$, in particular, $u$ is nontrivial if $\varliminf J_{\eps_j}(u_j) < 0$ or $\varlimsup J_{\eps_j}(u_j) > 0$.
\end{enumroman}
Moreover, $u$ satisfies the inequality $- \Delta u \le g(x,(u - 1)_+)$ in the distributional sense in $\Omega$ and the equation $- \Delta u = g(x,(u - 1)_+)$ in the classical sense in $\Omega \setminus F(u)$.
\end{theorem}

Theorem \ref{Theorem 1} will be proved in Section \ref{Regularization}. The following nondegeneracy estimate is needed to establish more detailed properties of the free boundary of the singular limit $u$ in this theorem.

\begin{definition} \label{Definition 2}
We say that a function $u \in C(\closure{\Omega})$ is nondegenerate if there exist constants $r_0, c > 0$ such that whenever $x_0 \in \set{u > 1}$ and $r := \dist{x_0}{\set{u \le 1}} \le r_0$,
\[
u(x_0) \ge 1 + cr.
\]
\end{definition}

If the limit $u$ is nondegenerate, then it is a weak viscosity solution in the following sense by the results of Lederman and Wolanski in \cite{MR2281453}.

\begin{definition} \label{Definition 4}
We say that a function $u \in C(\closure{\Omega})$ satisfies the free boundary condition
\[
|\nabla u^+|^2 - |\nabla u^-|^2 = 2
\]
in the weak viscosity sense if whenever there is a ball $B \subset \set{u > 1}$ tangent to $F(u)$ at a point $x_0$, either there are $\alpha > 0$ and $\beta > 0$ such that $\alpha^2 \le 2$, $\beta^2 \le 2$, and
\[
u(x) = 1 + \alpha\, {\ipa{x - x_0}{\nu}}_+ + \beta\, {\ipa{x - x_0}{\nu}}_- + \o(|x - x_0|) \quad \text{as } x \to x_0
\]
with $\nu$ the interior normal to $\bdry{B}$ at $x_0$, or else there are $\alpha > 0$ and $\beta \ge 0$ such that $\alpha^2 - \beta^2 = 2$ and
\[
u(x) = 1 + \alpha\, {\ipa{x - x_0}{\nu}}_+ - \beta\, {\ipa{x - x_0}{\nu}}_- + \o(|x - x_0|) \quad \text{as } x \to x_0.
\]
If the ball $B \subset \interior{\set{u \le 1}}$, then the second asymptotic formula holds with $\alpha$ and $\beta$ as above, but with $\nu$ the exterior normal to $\bdry{B}$ at $x_0$.
\end{definition}

The case where $u > 1$ on both sides of the free boundary can be ruled out and we can obtain a stronger form of viscosity solution if $u$ has the following positive density property.

\begin{definition} \label{Definition 3}
We say that a function $u \in C(\closure{\Omega})$ has the positive density property for $\set{u > 1}$ and $\set{u \le 1}$ if there exist constants $r_0, c > 0$ such that whenever $x_0 \in F(u)$ and $0 < r \le r_0$,
\[
c \le \frac{\vol{{\set{u > 1}} \cap B_r(x_0)}}{\vol{B_r(x_0)}} \le 1 - c.
\]
\end{definition}

The lower bound by $c$ in this definition follows from the nondegeneracy of Definition \ref{Definition 2}. The upper bound by $1 - c$ is a complementary nondegeneracy for the region $\set{u \le 1}$. We will prove the following proposition in the next section.

\begin{proposition} \label{Proposition 1}
If the limit $u$ in Theorem \ref{Theorem 1} is nondegenerate and has the positive density property for $\set{u > 1}$ and $\set{u \le 1}$, then it satisfies the free boundary condition in the viscosity sense, i.e., whenever there is a ball $B$ tangent to $F(u)$ at a point $x_0$, $u$ has an asymptotic expansion of the form
\[
u(x) = 1 + \alpha\, {\ipa{x - x_0}{\nu}}_+ - \beta\, {\ipa{x - x_0}{\nu}}_- + \o(|x - x_0|) \quad \text{as } x \to x_0
\]
with $\alpha > 0$, $\beta \ge 0$, and $\alpha^2 - \beta^2 = 2$, where $\nu$ is the interior unit normal to $\bdry{B}$ at $x_0$ if $B \subset \set{u > 1}$ and the exterior unit normal if $B \subset \interior{\set{u \le 1}}$.
\end{proposition}

We also need $u$ to be a variational solution in the following sense in order to establish regularity of the free boundary.

\begin{definition} \label{Definition 1}
We say that a function $u \in H^1_0(\Omega)$ satisfies the free boundary condition in the variational sense if
\begin{equation} \label{3.15}
\int_\Omega \left[\left(\half\, |\nabla u|^2 + \goodchi_{\set{u > 1}}(x) - G(x,(u - 1)_+)\right) \divg \Phi - \nabla u\, (D\Phi) \cdot \nabla u\right] dx = 0
\end{equation}
for all $\Phi \in C^1_0(\Omega,\R^N)$.
\end{definition}

The corresponding regularized equation
\begin{equation} \label{15}
\int_\Omega \left[\left(\half\, |\nabla u|^2 + \B\left(\frac{u - 1}{\eps}\right) - G(x,(u - 1)_+)\right) \divg \Phi - \nabla u\, (D\Phi) \cdot \nabla u\right] dx = 0
\end{equation}
is the critical point equation for $J_\eps$ with respect to domain variations. Indeed, the mapping $x \mapsto x - t\, \Phi(x)$ is a diffeomorphism of $\Omega$ for sufficiently small $t$, and the left-hand side of the equation \eqref{15} is
\[
\restr{\frac{d}{dt}}{t=0} J_\eps(u(x - t\, \Phi(x))).
\]
In the next section we will prove the following proposition showing that the convergence of the critical values is sufficient for the singular limit to be a variational solution.

\begin{proposition} \label{Proposition 2}
If $J_{\eps_j}(u_j) \to J(u)$ in Theorem \ref{Theorem 1}, then $u$ satisfies the free boundary condition in the variational sense.
\end{proposition}

Our main regularity result is the following theorem establishing full regularity of the free boundary in dimension $N = 2$ and partial regularity in higher dimensions.

\begin{theorem} \label{Theorem 4}
Assume $(g_1)$ and $(g_2)$. Let $u \in H^1_0(\Omega) \cap C^2(\closure{\Omega} \setminus F(u))$ be a nondegenerate and Lipschitz continuous solution of the equation $- \Delta u = g(x,(u - 1)_+)$ in $\Omega \setminus F(u)$. Assume further that $u$ has the positive density property for $\set{u > 1}$ and $\set{u \le 1}$, and satisfies the free boundary condition in the viscosity sense and in the variational sense. Then its free boundary $F(u)$ has finite $(N - 1)$-dimensional Hausdorff measure and is a $C^\infty$-hypersurface except on a closed set of Hausdorff dimension at most $N - 3$. In particular, $F(u)$ is smooth in dimension $N = 2$ and has at most finitely many nonsmooth points in dimension $N = 3$. Near the smooth subset of $F(u)$, $(u - 1)_\pm$ are smooth and the free boundary condition is satisfied in the classical sense.
\end{theorem}

This theorem is proved by the methods of Caffarelli \cite{MR861482,MR990856,MR1029856,MR973745}, Caffarelli and Salsa \cite{MR2145284}, Lederman and Wolanski \cite{MR2281453}, Jerison and Kamburov \cite{MR3567263}, and Weiss \cite{MR1759450}, who studied the case where $u$ is harmonic in $\Omega \setminus \bdry{\set{u > 0}}$. The details of the case where $u$ solves the inhomogeneous equation $- \Delta u = (u - 1)_+^{p-1}$ in $\Omega \setminus F(u)$, where $p > 2$ if $N = 2$ and $2 < p < 2^\ast$ if $N \ge 3$, are given in Jerison and Perera \cite{MR3790500}. There are no substantial differences between the proof in that case and the proof in the general case considered here.

To apply Theorem \ref{Theorem 4} to the singular limit $u$ in Theorem \ref{Theorem 1}, we have to show that $u$ is nondegenerate, has the positive density property, and that $J_{\eps_j}(u_j) \to J(u)$. Then $u$ satisfies the free boundary condition in the viscosity sense by Proposition \ref{Proposition 1} and in the variational sense by Proposition \ref{Proposition 2}, so the conclusions of Theorem \ref{Theorem 4} hold for $u$. In particular, the free boundary of $u$ has finite $(N - 1)$-dimensional Hausdorff measure and is smooth except on a closed set of Hausdorff dimension at most $N - 3$. We carry out this program for a class of superlinear free boundary problems next.

We assume that the nonlinearity $g$ satisfies, in addition to the conditions $(g_1)$ and $(g_2)$,
\begin{enumerate}
\item[$(g_3)$] $g(x,s) > 0$ for all $x \in \Omega$ and $s > 0$,
\item[$(g_4)$] there exists $\mu > 2$ such that the mappings
    \[
    s \mapsto \frac{g(x,s)}{s^{\mu - 1}}, \qquad s \mapsto \frac{1}{\mu}\, s g(x,s) - G(x,s)
    \]
    are nondecreasing for all $x \in \Omega$ and $s > 0$.
\end{enumerate}

Some examples of nonlinearities that satisfy the conditions $(g_1)$--$(g_4)$ are as follows:
\begin{enumarab}
\item For the subcritical pure power case $g(x,s) = s^{p-1}$, where $p > 2$ if $N = 2$ and $2 < p < 2^\ast$ if $N \ge 3$, considered in Jerison and Perera \cite{MR3790500}, $(g_1)$--$(g_4)$ hold with $\mu = p$.
\item A sum of powers
    \[
    g(x,s) = \sum_{i=1}^n s^{p_i - 1},
    \]
    where $p_i > 2$ if $N = 2$ and $2 < p_i \le 2^\ast$ if $N \ge 3$, satisfies $(g_1)$--$(g_4)$ with $\mu = \min p_i$.
\item Let $\mu > 2$ if $N = 2$ and $2 < \mu \le 2^\ast$ if $N \ge 3$, and let $a$ be a $C^1$-function on $\Omega \times [0,\infty)$ satisfying
    \begin{enumerate}
    \item[$(i)$] $a(x,s) > 0$ for all $x \in \Omega$ and $s > 0$,
    \item[$(ii)$] for some $a_3, a_4 > 0$,
        \[
        a(x,s) \le \begin{cases}
        a_3\, e^{a_4 s^2} & \text{if } N = 2\\[7.5pt]
        a_3 + a_4\, s^{2^\ast - \mu} & \text{if } N \ge 3
        \end{cases}
        \]
        for all $(x,s) \in \Omega \times [0,\infty)$,
    \item[$(iii)$] the mapping $s \mapsto a(x,s)$ is nondecreasing for all $(x,s) \in \Omega \times [0,\infty)$.
    \end{enumerate}
    Then $g(x,s) = a(x,s)\, s^{\mu - 1}$ satisfies $(g_1)$--$(g_4)$.
\end{enumarab}

The conditions $(g_3)$ and $(g_4)$ imply the well-known Ambrosetti-Rabinowitz superlinearity condition
\begin{equation} \label{103}
0 < \mu\, G(x,s) \le s g(x,s) \quad \forall (x,s) \in \Omega \times (0,\infty).
\end{equation}
This together with $(g_2)$ implies
\begin{equation} \label{1031}
G(x,s) \ge G(x,1)\, s^\mu - C \quad \forall (x,s) \in \Omega \times [0,\infty)
\end{equation}
for some constant $C > 0$, so $G$ grows superquadratically. The condition $(g_4)$ also implies that
\begin{equation} \label{104}
g(x,ts) \le t^{\mu - 1} g(x,s) \quad \text{if } t \in [0,1]
\end{equation}
and
\begin{equation} \label{105}
g(x,ts) \ge t^{\mu - 1} g(x,s) \quad \text{if } t \in [1,\infty)
\end{equation}
for all $(x,s) \in \Omega \times [0,\infty)$. When $g(x,\cdot) \in C^1([0,\infty))$ for all $x \in \Omega$, $(g_4)$ is equivalent to
\[
s\, \frac{\partial g}{\partial s}(x,s) \ge (\mu - 1)\, g(x,s) \quad \forall (x,s) \in \Omega \times [0,\infty).
\]

If $u$ is a solution of problem \eqref{1}, then $u$ is harmonic in $\Omega \setminus \closure{\set{u > 1}}$ by $(g_1)$, so $u$ is either positive everywhere or vanishes identically by the maximum principle. If $u \le 1$ everywhere, then $u$ is harmonic in $\Omega$ and hence vanishes identically again. So if $u$ is a nontrivial solution, then $u > 0$ in $\Omega$ and $u > 1$ in a nonempty open subset of $\Omega$, where it satisfies the equation $- \Delta u = g(x,u - 1)$. Multiplying this equation by $u - 1$ and integrating over the set $\set{u > 1}$ shows that $u$ satisfies the integral identity
\[
\int_\Omega |\nabla u^+|^2\, dx = \int_\Omega u^+ g(x,u^+)\, dx,
\]
where $u^+ = (u - 1)_+$. Thus, setting $\W = \bigset{u \in H^1_0(\Omega) : u^+ \ne 0}$, all nontrivial solutions of problem \eqref{1} lie on the set
\[
\M = \set{u \in \W : \int_\Omega |\nabla u^+|^2\, dx = \int_\Omega u^+ g(x,u^+)\, dx}.
\]
This set will play an important role in our study of superlinear free boundary problems. By $(g_3)$ and $(g_4)$,
\[
J(u) > \int_\Omega \left[\half\, |\nabla u^+|^2 - G(x,u^+)\right] dx > \int_\Omega \left[\frac{1}{\mu}\, u^+ g(x,u^+) - G(x,u^+)\right] dx \ge 0
\]
for all $u \in \M$.

We will see that \eqref{1031} and \eqref{104} imply that the functional $J$ has the mountain pass geometry. Let
\[
\Gamma = \set{\gamma \in C([0,1],H^1_0(\Omega)) : \gamma(0) = 0,\, J(\gamma(1)) < 0}
\]
be the class of paths joining the origin to the set $\set{u \in H^1_0(\Omega) : J(u) < 0}$, and let
\begin{equation} \label{301}
c := \inf_{\gamma \in \Gamma}\, \max_{u \in \gamma([0,1])}\, J(u) > 0
\end{equation}
be the mountain pass level. Our goal is to prove the existence and regularity of a nonminimizing solution of mountain pass type at this level. We recall that $u \in H^1_0(\Omega)$ is a mountain pass point of $J$ if the set $\set{v \in V : J(v) < J(u)}$ is neither empty nor path connected for any neighborhood $V$ of $u$ (see Hofer \cite{MR812787}). The following proposition, proved exactly as in Jerison and Perera \cite[Proposition 4.2]{MR3790500}, shows that if $c = \inf_\M J$ and $u$ is a minimizer of $\restr{J}{\M}$, then $u$ is a mountain pass point.

\begin{proposition} \label{Proposition 4}
We have
\[
c \le \inf_{v \in \M}\, J(v).
\]
If $u \in \M$ and $J(u) = c$, then $u$ is a mountain pass point of $J$.
\end{proposition}

Let $\eps_j \searrow 0$. We will show that each approximating functional $J_{\eps_j}$ has a critical point $u_j$ of mountain pass type, apply Theorem \ref{Theorem 1}, and show that the singular limit $u$ is in $\M$ and satisfies
\[
J(u) = c = \inf_{v \in \M}\, J(v).
\]
Proposition \ref{Proposition 4} will then show that $u$ is a mountain pass point of $J$, and Proposition \ref{Proposition 3} bellow, proved in Section \ref{Nondegeneracy}, will give us the nondegeneracy estimates needed to apply Theorem \ref{Theorem 4}.

\begin{proposition} \label{Proposition 3}
Assume $(g_1)$--$(g_4)$. Let $u \in H^1_0(\Omega) \cap C^2(\closure{\Omega} \setminus F(u))$ be a Lipschitz continuous minimizer of $\restr{J}{\M}$ that satisfies the equation $- \Delta u = g(x,(u - 1)_+)$ in $\Omega \setminus F(u)$. Then $u$ is nondegenerate. If, in addition, $u$ satisfies the inequality $- \Delta u \le g(x,(u - 1)_+)$ in the distributional sense in $\Omega$, then $u$ has the positive density property for $\set{u > 1}$ and $\set{u \le 1}$.
\end{proposition}

First we consider the subcritical case where $(g_2)$ is replaced with the more restrictive growth condition
\begin{enumerate}
\item[$(g_2')$] for some $a_1, a_2 > 0$,
    \[
    g(x,s) \le a_1 + a_2\, s^{p-1}
    \]
    for all $(x,s) \in \Omega \times [0,\infty)$, where $p > 2$ if $N = 2$ and $2 < p < 2^\ast$ if $N \ge 3$.
\end{enumerate}
We note that $p \ge \mu$ by \eqref{1031}.

For example, let $p > \mu > 2$ if $N = 2$ and $2 < \mu < p < 2^\ast$ if $N \ge 3$, and let $a$ be a $C^1$-function on $\Omega \times [0,\infty)$ satisfying
\begin{enumerate}
\item[$(i)$] $a(x,s) > 0$ for all $x \in \Omega$ and $s > 0$,
\item[$(ii)$] for some $a_3, a_4 > 0$,
    \[
    a(x,s) \le a_3 + a_4\, s^{p - \mu}
    \]
    for all $(x,s) \in \Omega \times [0,\infty)$,
\item[$(iii)$] the mapping $s \mapsto a(x,s)$ is nondecreasing for all $(x,s) \in \Omega \times [0,\infty)$.
\end{enumerate}
Then $g(x,s) = a(x,s)\, s^{\mu - 1}$ satisfies $(g_1)$, $(g_2')$, $(g_3)$, and $(g_4)$.

Our main result in the subcritical case is the following theorem.

\begin{theorem} \label{Theorem 2}
Assume $(g_1)$, $(g_2')$, $(g_3)$, and $(g_4)$. Then problem \eqref{1} has a Lipschitz continuous mountain pass solution $u \in H^1_0(\Omega) \cap C^2(\closure{\Omega} \setminus F(u))$ at the level $c$ that satisfies the equation $- \Delta u = g(x,(u - 1)_+)$ in the classical sense in $\Omega \setminus F(u)$. The free boundary condition is satisfied in the viscosity sense and in the variational sense. The free boundary $F(u)$ has finite $(N - 1)$-dimensional Hausdorff measure and is a $C^\infty$-hypersurface except on a closed set of Hausdorff dimension at most $N - 3$. Near the smooth subset of $F(u)$, $(u - 1)_\pm$ are smooth and the free boundary condition is satisfied in the classical sense.
\end{theorem}

This theorem will be proved in Section \ref{Subcritical}. Finally we let $N \ge 3$ and consider the critical problem
\begin{equation} \label{500}
\left\{\begin{aligned}
- \Delta u & = \kappa\, (u - 1)_+^{2^\ast - 1} + \lambda\, (u - 1)_+^{\mu - 1} && \text{in } \Omega \setminus F(u)\\[10pt]
|\nabla u^+|^2 - |\nabla u^-|^2 & = 2 && \text{on } F(u)\\[10pt]
u & = 0 && \text{on } \bdry{\Omega},
\end{aligned}\right.
\end{equation}
where $\kappa, \lambda > 0$ are parameters and $2 < \mu < 2^\ast$. Let
\[
J(u) = \int_\Omega \left[\half\, |\nabla u|^2 + \goodchi_{\set{u > 1}}(x) - \frac{\kappa}{2^\ast}\, (u - 1)_+^{2^\ast} - \frac{\lambda}{\mu}\, (u - 1)_+^\mu\right] dx, \quad u \in H^1_0(\Omega)
\]
be the associated variational functional and let
\[
c := \inf_{\gamma \in \Gamma}\, \max_{u \in \gamma([0,1])}\, J(u) > 0,
\]
where $\Gamma = \set{\gamma \in C([0,1],H^1_0(\Omega)) : \gamma(0) = 0,\, J(\gamma(1)) < 0}$, be its mountain pass level. Our main result for this problem is the following theorem.

\begin{theorem} \label{Theorem 3}
Let $N \ge 3$ and assume that $2 < \mu < 2^\ast$. Then there exists a $\kappa_\ast > 0$ such that for $0 < \kappa < \kappa_\ast$ and all $\lambda > 0$, problem \eqref{500} has a Lipschitz continuous mountain pass solution $u \in H^1_0(\Omega) \cap C^2(\closure{\Omega} \setminus F(u))$ at the level $c$ that satisfies the equation $- \Delta u = \lambda\, (u - 1)_+^{\mu - 1} + \kappa\, (u - 1)_+^{2^\ast - 1}$ in the classical sense in $\Omega \setminus F(u)$. The free boundary condition is satisfied in the viscosity sense and in the variational sense. The free boundary $F(u)$ has finite $(N - 1)$-dimensional Hausdorff measure and is a $C^\infty$-hypersurface except on a closed set of Hausdorff dimension at most $N - 3$. Near the smooth subset of $F(u)$, $(u - 1)_\pm$ are smooth and the free boundary condition is satisfied in the classical sense.
\end{theorem}

This theorem will be proved in Section \ref{Critical}. The limiting case $\mu = 2$ was considered in Yang and Perera \cite{MR3872792}, where a nondegenerate mountain pass solution that satisfies the free boundary condition in the viscosity sense was obtained for sufficiently small $\kappa > 0$ and $\lambda > \lambda_1$, the first Dirichlet eigenvalue of $- \Delta$ in $\Omega$. However, the question of regularity of the free boundary was not considered in \cite{MR3872792}.

\begin{notation}
Throughout the paper we write
\[
u^+ = (u - 1)_+, \qquad u^- = u - u^+.
\]
\end{notation}

\section{Regularization} \label{Regularization}

In this section we prove Theorem \ref{Theorem 1} and Propositions \ref{Proposition 1} and \ref{Proposition 2}.

The crucial ingredient in the passage to the limit in the proof of Theorem \ref{Theorem 1} is the following uniform Lipschitz continuity result of Caffarelli et al.\! \cite{MR1906591}.

\begin{lemma}[{\cite[Theorem 5.1]{MR1906591}}] \label{Lemma 6}
Let $u$ be a Lipschitz continuous function on $B_1(0) \subset \R^N$ satisfying the distributional inequalities
\[
\pm \Delta u \le A \left(\frac{1}{\eps}\, \goodchi_{\set{|u - 1| < \eps}}(x) + 1\right)
\]
for some constants $A > 0$ and $0 < \eps \le 1$. Then there exists a constant $C > 0$, depending on $N$, $A$, and $\dint_{B_1(0)} u^2\, dx$, but not on $\eps$, such that
\[
\max_{x \in B_{1/2}(0)}\, |\nabla u(x)| \le C.
\]
\end{lemma}

\begin{proof}[Proof of Theorem \ref{Theorem 1}]
We may assume that $0 < \eps_j \le 1$. Since $\seq{u_j}$ is bounded in $L^\infty(\Omega)$, $|g(x,(u_j - 1)_+)| \le A_0$ for some constant $A_0 > 0$ by $(g_2)$. Let $\varphi_0 > 0$ be the solution of
\[
\left\{\begin{aligned}
- \Delta \varphi_0 & = A_0 && \text{in } \Omega\\[10pt]
\varphi_0 & = 0 && \text{on } \bdry{\Omega}.
\end{aligned}\right.
\]
Since $\beta \ge 0$, $- \Delta u_j \le A_0$ in $\Omega$, and hence
\begin{equation} \label{3.4}
0 \le u_j(x) \le \varphi_0(x) \quad \forall x \in \Omega
\end{equation}
by the maximum principle. The majorant $\varphi_0$ gives a uniform lower bound $\delta_0 > 0$ on the distance from the set $\set{u_j \ge 1}$ to $\bdry{\Omega}$. Since $u_j$ is positive, harmonic, and bounded by $1$ in a $\delta_0$ neighborhood of $\bdry{\Omega}$, it follows from standard boundary regularity theory that the sequence $\seq{u_j}$ is bounded in the $C^{2,\alpha}$ norm, and hence compact in the $C^2$ norm, in a $\delta_0/2$ neighborhood.

Since $0 \le \beta \le 2\, \goodchi_{(-1,1)}$,
\[
\pm \Delta u_j = \pm \frac{1}{\eps_j}\, \beta\left(\frac{u_j - 1}{\eps_j}\right) \mp g(x,(u_j - 1)_+) \le \frac{2}{\eps_j}\, \goodchi_{\set{|u_j - 1| < \eps_j}}(x) + A_0.
\]
Since $\seq{u_j}$ is bounded in $L^2(\Omega)$, it follows from Lemma \ref{Lemma 6} that there exists a constant $C > 0$ such that
\begin{equation} \label{3.2}
\max_{x \in B_{r/2}(x_0)}\, |\nabla u_j(x)| \le \frac{C}{r}
\end{equation}
whenever $r > 0$ and $B_r(x_0) \subset \Omega$. Hence $u_j$ is uniformly Lipschitz continuous on the compact subset of $\Omega$ at distance greater or equal to $\delta_0/2$ from $\bdry{\Omega}$.

Thus, a renamed subsequence of $\seq{u_j}$ converges uniformly on $\closure{\Omega}$ to a Lipschitz continuous function $u$ with zero boundary values, with strong convergence in $C^2$ on a $\delta_0/2$ neighborhood of $\bdry{\Omega}$. Since $\seq{u_j}$ is bounded in $H^1_0(\Omega)$, a further subsequence converges weakly in $H^1_0(\Omega)$ to $u$.

Since $\beta \ge 0$, testing the equation
\begin{equation} \label{3.5}
- \Delta u_j = - \frac{1}{\eps_j}\, \beta\left(\frac{u_j - 1}{\eps_j}\right) + g(x,(u_j - 1)_+)
\end{equation}
with any nonnegative test function and passing to the limit shows that
\begin{equation} \label{3.7}
- \Delta u \le g(x,(u - 1)_+)
\end{equation}
in the distributional sense.

Next we show that $u$ satisfies the equation $- \Delta u = g(x,(u - 1)_+)$ in the set $\set{u \ne 1}$. Let $\varphi \in C^\infty_0(\set{u > 1})$. Then $u \ge 1 + 2\, \eps$ on the support of $\varphi$ for some $\eps > 0$. For all sufficiently large $j$, $\eps_j < \eps$ and $|u_j - u| < \eps$ in $\Omega$, so $u_j \ge 1 + \eps_j$ on the support of $\varphi$. So testing \eqref{3.5} with $\varphi$ gives
\[
\int_\Omega \nabla u_j \cdot \nabla \varphi\, dx = \int_\Omega g(x,(u_j - 1)_+)\, \varphi\, dx,
\]
and passing to the limit gives
\begin{equation} \label{3.6}
\int_\Omega \nabla u \cdot \nabla \varphi\, dx = \int_\Omega g(x,(u - 1)_+)\, \varphi\, dx
\end{equation}
since $u_j$ converges to $u$ weakly in $H^1_0(\Omega)$ and uniformly on $\Omega$. Hence $u$ is a distributional, and hence a classical, solution of $- \Delta u = g(x,(u - 1)_+)$ in the set $\set{u > 1}$. A similar argument shows that $u$ satisfies $\Delta u = 0$ in the set $\set{u < 1}$.

Now we show that $u$ is also harmonic in the possibly larger set $\interior{\set{u \le 1}}$. Since $u$ is harmonic in $\set{u < 1}$, $\min\, (u,1)$ satisfies the super-mean value property. This implies that
\begin{equation} \label{3.8}
- \Delta \min\, (u,1) \ge 0
\end{equation}
in the distributional sense (see, e.g., Alt and Caffarelli \cite[Remark 4.2]{MR618549}). It follows from \eqref{3.8}, \eqref{3.7}, and $(g_1)$ that $\Delta u = 0$ as a distribution in $\interior{\set {u \le 1}}$, and hence also in the classical sense.

Since $u_j \rightharpoonup u$ in $H^1_0(\Omega)$, $\norm{u} \le \varliminf \norm{u_j}$, so it suffices to show that $\varlimsup \norm{u_j} \le \norm{u}$ to prove \ref{Theorem 1.ii}. Recall that $u_j$ converges in the $C^2$ norm to $u$ in a neighborhood of $\bdry{\Omega}$ in $\closure{\Omega}$. Multiplying \eqref{3.5} by $u_j - 1$, integrating by parts, and noting that $\beta((t - 1)/\eps_j)\, (t - 1) \ge 0$ for all $t$ gives
\begin{multline} \label{3.9}
\int_\Omega |\nabla u_j|^2\, dx \le \int_\Omega g(x,(u_j - 1)_+)\, (u_j - 1)\, dx - \int_{\bdry{\Omega}} \frac{\partial u_j}{\partial n}\, d\sigma\\[7.5pt]
\to \int_\Omega g(x,(u - 1)_+)\, (u - 1)\, dx - \int_{\bdry{\Omega}} \frac{\partial u}{\partial n}\, d\sigma,
\end{multline}
where $n$ is the outward unit normal to $\bdry{\Omega}$. Fix $0 < \eps < 1$. Taking $\varphi = (u - 1 - \eps)_+$ in \eqref{3.6} gives
\begin{equation} \label{3.10}
\int_{\set{u > 1 + \eps}} |\nabla u|^2\, dx = \int_\Omega g(x,(u - 1)_+)\, (u - 1 - \eps)_+\, dx,
\end{equation}
and integrating $(u - 1 + \eps)_-\, \Delta u = 0$ over $\Omega$ yields
\begin{equation} \label{3.11}
\int_{\set{u < 1 - \eps}} |\nabla u|^2\, dx = - (1 - \eps) \int_{\bdry{\Omega}} \frac{\partial u}{\partial n}\, d\sigma.
\end{equation}
Adding \eqref{3.10} and \eqref{3.11}, and letting $\eps \searrow 0$ gives
\[
\int_\Omega |\nabla u|^2\, dx = \int_\Omega g(x,(u - 1)_+)\, (u - 1)\, dx - \int_{\bdry{\Omega}} \frac{\partial u}{\partial n}\, d\sigma
\]
since $\dint_{\set{u = 1}} |\nabla u|^2\, dx = 0$ and $(g_1)$ holds. This together with \eqref{3.9} gives
\[
\varlimsup \int_\Omega |\nabla u_j|^2\, dx \le \int_\Omega |\nabla u|^2\, dx
\]
as desired.

To prove \ref{Theorem 1.iii}, write
\begin{multline*}
J_{\eps_j}(u_j) = \int_\Omega \left[\half\, |\nabla u_j|^2 + \B\left(\frac{u_j - 1}{\eps_j}\right) \goodchi_{\set{u \ne 1}}(x) - G(x,(u_j - 1)_+)\right] dx\\[7.5pt]
+ \int_{\set{u = 1}} \B\left(\frac{u_j - 1}{\eps_j}\right) dx.
\end{multline*}
Since $\B((u_j - 1)/\eps_j)\, \goodchi_{\set{u \ne 1}}$ converges pointwise to $\goodchi_{\set{u > 1}}$ and is bounded by $1$, the first integral converges to $J(u)$ by \ref{Theorem 1.i} and \ref{Theorem 1.ii}, and
\[
0 \le \int_{\set{u = 1}} \B\left(\frac{u_j - 1}{\eps_j}\right) dx \le \vol{{\set{u = 1}}}
\]
since $0 \le \B \le 1$, so \ref{Theorem 1.iii} follows.
\end{proof}

Next we prove Proposition \ref{Proposition 1}.

\begin{proof}[Proof of Proposition \ref{Proposition 1}]
Let
\[
f_j(x) = - g(x,(u_j(x) - 1)_+), \quad f(x) = - g(x,(u(x) - 1)_+), \quad x \in \Omega,
\]
and note that $f_j$ converges uniformly to $f$ since $u_j$ converges uniformly to $u$. Since $u_j$ solves \eqref{13} with $\eps = \eps_j$ and $u$ is nondegenerate, Corollaries 7.1 and 7.2 of Lederman and Wolanski \cite{MR2281453} then imply that $u$ satisfies the free boundary condition in the weak viscosity sense. The fact that $\vol{{\set{u \le 1}} \cap B_r(x_0)} \ge cr^N$ rules out the case $\beta < 0$.
\end{proof}

In preparation for the proof of Proposition \ref{Proposition 2}, next we show that $u_j$ is a variational solution of the regularized problem \eqref{13} with $\eps = \eps_j$.

\begin{lemma} \label{Lemma 3}
If $u \in H^1_0(\Omega)$ is a critical point of $J_\eps$, then the equation \eqref{15} holds for all $\Phi \in C^1_0(\Omega,\R^N)$.
\end{lemma}

\begin{proof}
For $\Phi \in C^1_0(\Omega,\R^N)$,
\begin{eqnarray*}
& & \divg \left[\left(\half\, |\nabla u|^2 + \B\left(\frac{u - 1}{\eps}\right) - G(x,(u - 1)_+)\right) \Phi - (\nabla u \cdot \Phi)\, \nabla u\right]\\[7.5pt]
& = & \left(\half\, |\nabla u|^2 + \B\left(\frac{u - 1}{\eps}\right) - G(x,(u - 1)_+)\right) \divg \Phi - \nabla u\, (D\Phi) \cdot \nabla u\\[7.5pt]
& & + \left(- \Delta u + \frac{1}{\eps}\, \beta\left(\frac{u - 1}{\eps}\right) - g(x,(u - 1)_+)\right)(\nabla u \cdot \Phi).
\end{eqnarray*}
Since $u$ satisfies \eqref{13}, integrating over $\Omega$ using the divergence theorem and noting that $\Phi = 0$ on $\bdry{\Omega}$ gives \eqref{15}.
\end{proof}

We are now ready to prove Proposition \ref{Proposition 2}.

\begin{proof}[Proof of Proposition \ref{Proposition 2}]
Let $\Phi \in C^1_0(\Omega,\R^N)$. Then
\begin{equation} \label{3.14}
\int_\Omega \left[\left(\half\, |\nabla u_j|^2 + \B\left(\frac{u_j - 1}{\eps_j}\right) - G(x,(u_j - 1)_+)\right) \divg \Phi - \nabla u_j\, (D\Phi) \cdot \nabla u_j\right] dx = 0
\end{equation}
for all $j$ by Lemma \ref{Lemma 3}. Since $J_{\eps_j}(u_j) \to J(u)$, and \ref{Theorem 1.i} and \ref{Theorem 1.ii} hold,
\[
\lim \int_\Omega \B\left(\frac{u_j - 1}{\eps_j}\right) dx = \vol{{\set{u > 1}}},
\]
and hence
\begin{equation} \label{3.12}
\varlimsup \int_{\set{u \le 1}} \B\left(\frac{u_j - 1}{\eps_j}\right) dx \le \vol{{\set{u > 1}}} - \varliminf \int_{\set{u > 1}} \B\left(\frac{u_j - 1}{\eps_j}\right) dx.
\end{equation}
On the other hand,
\[
\varliminf \int_{\set{u > 1}} \B\left(\frac{u_j - 1}{\eps_j}\right) dx \ge \vol{{\set{u \ge 1 + \delta}}}
\]
for all $\delta > 0$ by \ref{Theorem 1.i}, and letting $\delta \searrow 0$ gives
\begin{equation} \label{3.13}
\varliminf \int_{\set{u > 1}} \B\left(\frac{u_j - 1}{\eps_j}\right) dx \ge \vol{{\set{u > 1}}}.
\end{equation}
It follows from \eqref{3.12} and \eqref{3.13} that
\[
\lim \int_{\set{u \le 1}} \B\left(\frac{u_j - 1}{\eps_j}\right) dx = 0,
\]
so
\[
\lim \int_\Omega \B\left(\frac{u_j - 1}{\eps_j}\right) \divg \Phi\, dx = \lim \int_{\set{u > 1}} \B\left(\frac{u_j - 1}{\eps_j}\right) \divg \Phi\, dx = \int_{\set{u > 1}} \divg \Phi\, dx
\]
by the dominated convergence theorem. Now passing to the limit in \eqref{3.14} using \ref{Theorem 1.i} and \ref{Theorem 1.ii} gives \eqref{3.15}.
\end{proof}

\section{Nondegeneracy} \label{Nondegeneracy}

In this section we prove Proposition \ref{Proposition 3}. Note that $u^- \ne 0$ for $u \in \W$. Let $\zeta_u : [-1,\infty) \to H^1_0(\Omega)$ be the path
\[
\zeta_u(t) = \begin{cases}
(1 + t) u^- & \text{if } t \in [-1,0]\\[7.5pt]
u^- + t u^+ & \text{if } t \in (0,\infty),
\end{cases}
\]
which joins the origin to ``infinity'' and passes through $u$ at $t = 1$, and let
\[
\varphi_u(t) = J(\zeta_u(t)).
\]
For $t \in [-1,0]$,
\[
\varphi_u(t) = \frac{(1 + t)^2}{2} \int_\Omega |\nabla u^-|^2\, dx
\]
and is strictly increasing in $t$. For $t \in (0,\infty)$,
\begin{equation} \label{111}
\varphi_u(t) = \int_\Omega \left[\half\, |\nabla u^-|^2 + \frac{t^2}{2}\, |\nabla u^+|^2 - G(x,t u^+)\right] dx + \vol{{\set{u > 1}}}
\end{equation}
and
\begin{equation} \label{2}
\varphi_u'(t) = \int_\Omega \left(t\, |\nabla u^+|^2 - u^+ g(x,t u^+)\right) dx.
\end{equation}
So
\[
\lim_{t \searrow 0}\, \varphi_u(t) = \varphi_u(0) + \vol{{\set{u > 1}}}
\]
and hence $\varphi_u$ has a jump equal to $\vol{{\set{u > 1}}}$ at $t = 0$, and $\zeta_u(t)$ lies on $\M$ if and only if $t > 0$ and $\varphi_u'(t) = 0$, in particular, $u \in \M$ if and only if $\varphi_u'(1) = 0$.

Writing \eqref{2} as
\[
\varphi_u'(t) = t \left[\int_\Omega |\nabla u^+|^2\, dx - \int_\Omega \frac{u^+}{t}\, g(x,t u^+)\, dx\right],
\]
the last integral is strictly increasing in $t$ by $(g_4)$, goes to zero as $t \searrow 0$ by \eqref{104}, and goes to infinity as $t \to \infty$ by \eqref{105}, so there exists a $t_u > 0$ such that $\varphi_u'(t) > 0$ for $t \in (0,t_u)$, $\varphi_u'(t_u) = 0$, and $\varphi_u'(t) < 0$ for $t \in (t_u,\infty)$. So
\[
\M = \set{\zeta_u(t_u) : u \in \W},
\]
and $\varphi_u(t)$ is strictly increasing for $t \in (0,t_u)$, attains its maximum at $t = t_u$, is strictly decreasing for $t \in (t_u,\infty)$, and goes to $- \infty$ as $t \to \infty$ by \eqref{1031}.

For $u \in \W$, the path $\zeta_u$ intersects $\M$ exactly at one point, namely, where $t = t_u$. If $u \in \M$, then $t_u = 1$ and hence $\zeta_u(t_u) = u$. So we can define a (nonradial) continuous projection $\pi : \W \to \M$ by
\[
\pi(u) = \zeta_u(t_u) = u^- + t_u u^+.
\]
By \eqref{111},
\begin{equation} \label{112}
J(\pi(u)) = \int_\Omega \left[\half\, |\nabla u^-|^2 + \frac{t_u^2}{2}\, |\nabla u^+|^2 - G(x,t_u u^+)\right] dx + \vol{{\set{u > 1}}}.
\end{equation}
We have the following estimates for $t_u$.

\begin{lemma} \label{Lemma 2}
If $\dint_\Omega |\nabla u^+|^2\, dx \le \dint_\Omega u^+ g(x,u^+)\, dx$, then
\begin{equation} \label{3}
\frac{\dint_\Omega |\nabla u^+|^2\, dx}{\dint_\Omega u^+ g(x,u^+)\, dx} \le t_u^{\mu - 2} \le 1,
\end{equation}
and if $\dint_\Omega |\nabla u^+|^2\, dx \ge \dint_\Omega u^+ g(x,u^+)\, dx$, then
\begin{equation} \label{4}
1 \le t_u^{\mu - 2} \le \frac{\dint_\Omega |\nabla u^+|^2\, dx}{\dint_\Omega u^+ g(x,u^+)\, dx}.
\end{equation}
\end{lemma}

\begin{proof}
Since $\varphi_u'(t_u) = 0$, \eqref{2} gives
\begin{equation} \label{5}
t_u \int_\Omega |\nabla u^+|^2\, dx = \int_\Omega u^+ g(x,t_u u^+)\, dx.
\end{equation}
If $\dint_\Omega |\nabla u^+|^2\, dx \le \dint_\Omega u^+ g(x,u^+)\, dx$, then $\varphi_u'(1) \le 0$ and hence $t_u \le 1$. Then
\[
\int_\Omega u^+ g(x,t_u u^+)\, dx \le t_u^{\mu - 1} \int_\Omega u^+ g(x,u^+)\, dx
\]
by \eqref{104}, which together with \eqref{5} gives \eqref{3}. The proof of \eqref{4} is similar.
\end{proof}

We are now ready to prove the first part of Proposition \ref{Proposition 3}.

\begin{proof}[Proof of the nondegeneracy in Proposition \ref{Proposition 3}]
Assume that $r \le 1$, $B_r(x_0) \subset \set{u > 1}$, and $\exists\, x_1 \in \bdry{B_r(x_0)}$ such that $u(x_1) = 1$. Setting
\[
v(y) = \frac{1}{r}\, (u(x_0 + ry) - 1), \quad y \in B_1(0),
\]
we will show that there exists a constant $c > 0$ such that
\begin{equation} \label{6}
\alpha := v(0) = \frac{1}{r}\, (u(x_0) - 1) \ge c.
\end{equation}

We have
\[
- \Delta v = rg(x_0 + ry,rv) \quad \text{in } B_1(0)
\]
and
\begin{equation} \label{7}
0 < v(y) = \frac{1}{r}\, (u(x_0 + ry) - u(x_1)) \le \frac{L}{r}\, |x_0 + ry - x_1| \le 2L \quad \forall y \in B_1(0),
\end{equation}
where $L > 0$ is the Lipschitz constant of $u$. Let $h$ be the solution of the problem
\[
\left\{\begin{aligned}
- \Delta h & = rg(x_0 + ry,rv(y)) && \text{in } B_1(0)\\[10pt]
h & = 0 && \text{on } \bdry{B_1(0)}.
\end{aligned}\right.
\]
Since $g(x_0 + ry,rv)$ is nonnegative and bounded by $(g_1)$--$(g_3)$ and \eqref{7}, $0 \le h \le c_1 r$ for some $c_1 > 0$. So applying the Harnack inequality to $v - h + c_1 r$, there exists a constant $c_2 > 0$ such that
\[
v(y) \le c_2 (\alpha + r) \quad \forall y \in B_{2/3}(0).
\]

Take a smooth cutoff function $\psi : B_1(0) \to [0,1]$ such that $\psi = 0$ in $\closure{B_{1/3}(0)}$, $0 < \psi < 1$ in $B_{2/3}(0) \setminus \closure{B_{1/3}(0)}$, and $\psi = 1$ in $B_1(0) \setminus B_{2/3}(0)$, let
\[
w(y) = \begin{cases}
\min\, (v(y),c_2 (\alpha + r) \psi(y)) & \text{if } y \in B_{2/3}(0)\\[7.5pt]
v(y) & \text{otherwise},
\end{cases}
\]
and set
\[
z(x) = 1 + rw\left(\frac{x - x_0}{r}\right).
\]
Since $u$ is a minimizer of $\restr{J}{\M}$,
\[
J(u) \le J(\pi(z)).
\]
Since $z^- = u^-$, $z = 1$ in $\closure{B_{r/3}(x_0)}$, and $\set{z > 1} = \set{u > 1} \setminus \closure{B_{r/3}(x_0)}$, this inequality reduces to
\begin{equation} \label{102}
\int_\Omega \left[\half\, |\nabla u^+|^2 - G(x,u^+)\right] dx + \vol{B_{r/3}(x_0)} \le \int_\Omega \left[\frac{t_z^2}{2}\, |\nabla z^+|^2 - G(x,t_z z^+)\right] dx.
\end{equation}
Since $u, \pi(z) \in \M$,
\[
\int_\Omega |\nabla u^+|^2\, dx = \int_\Omega u^+ g(x,u^+)\, dx =: A
\]
and
\[
\int_\Omega t_z^2\, |\nabla z^+|^2\, dx = \int_\Omega t_z z^+ g(x,t_z z^+)\, dx,
\]
so the last inequality implies
\[
\int_\Omega \left[\half\, u^+ g(x,u^+) - G(x,u^+)\right] dx < \int_\Omega \left[\half\, t_z z^+ g(x,t_z z^+) - G(x,t_z z^+)\right] dx.
\]
Since the mapping $s \mapsto \ds{\half}\, s g(x,s) - G(x,s)$ is nondecreasing by $(g_4)$ and $z^+ \le u^+$, this leads to a contradiction if $t_z \le 1$, so $t_z > 1$.

Let $y = (x - x_0)/r$ and let
\[
\D = \set{x \in B_{2r/3}(x_0) : v(y) > c_2 (\alpha + r) \psi(y)}.
\]
Then $z = u$ outside $\D$, so the inequality \eqref{102} together with the fact that $G(x,s) \ge 0$ by $(g_3)$ implies
\begin{multline} \label{8}
\frac{t_z^2}{2} \int_\D |\nabla z^+|^2\, dx + \frac{t_z^2 - 1}{2} \int_{\Omega \setminus \D} |\nabla u^+|^2\, dx + \int_\D G(x,u^+)\, dx\\[7.5pt]
+ \int_{\Omega \setminus \D} \left[G(x,u^+) - G(x,t_z u^+)\right] dx \ge \vol{B_{1/3}(0)}\, r^N.
\end{multline}
We will show that \eqref{6} follows from this inequality if $r > 0$ is sufficiently small.

First we note some estimates. We have
\begin{equation} \label{9}
\int_\D |\nabla z^+|^2\, dx = c_2^2\, (\alpha + r)^2\, r^N \int_{\set{y : x \in \D}} |\nabla \psi|^2\, dy = \O(r^N)
\end{equation}
since $0 < \alpha \le 2L$ by \eqref{7}. Since $u^+ \le 2Lr$ in $B_r(x_0)$ by \eqref{7},
\begin{equation} \label{12}
\int_\D u^+ g(x,u^+)\, dx = \O(r^{\mu + N})
\end{equation}
by \eqref{104}, and hence
\begin{equation} \label{106}
\int_\D G(x,u^+)\, dx = \O(r^{\mu + N})
\end{equation}
by \eqref{103}.

Since $t_z > 1$,
\[
t_z^{\mu - 2} \le \frac{\dint_\Omega |\nabla z^+|^2\, dx}{\dint_\Omega z^+ g(x,z^+)\, dx}
\]
by Lemma \ref{Lemma 2}. We have
\[
\int_\Omega |\nabla z^+|^2\, dx = \int_{\Omega \setminus \D} |\nabla u^+|^2\, dx + \int_\D |\nabla z^+|^2\, dx \le A + \int_\D |\nabla z^+|^2\, dx,
\]
and since $g(x,s) \ge 0$ by $(g_1)$ and $(g_3)$,
\begin{multline*}
\int_\Omega z^+ g(x,z^+)\, dx = \int_{\Omega \setminus \D} u^+ g(x,u^+)\, dx + \int_\D z^+ g(x,z^+)\, dx\\[7.5pt]
\ge \int_\Omega u^+ g(x,u^+)\, dx - \int_\D u^+ g(x,u^+)\, dx \ge A - c_3\, r^{\mu + N}
\end{multline*}
for some $c_3 > 0$ by \eqref{12}. In view of \eqref{9}, it follows that
\[
t_z^2 \le 1 + \frac{2}{(\mu - 2) A} \int_\D |\nabla z^+|^2\, dx + c_4\, r^{\nu + N}
\]
for some $c_4 > 0$, where $\nu = \min\, (\mu,N) \ge 2$. Since $\dint_{\Omega \setminus \D} |\nabla u^+|^2\, dx \le A$, \eqref{106} holds, and $G(x,s)$ is increasing in $s$ by $(g_3)$, the inequality \eqref{8} now implies
\[
\frac{\mu}{(\mu - 2)\, r^N} \int_\D |\nabla z^+|^2\, dx + \O(r^\nu) \ge 2 \vol{B_{1/3}(0)}.
\]
This together with the first equality in \eqref{9} gives \eqref{6} for sufficiently small $r$.
\end{proof}

In preparation for the proof of the second part of Proposition \ref{Proposition 3}, let $L > 0$ denote the Lipschitz constant of $u$ in $\closure{\Omega}$ and let
\[
\delta_0 = \dist{\set{u > 1}}{\bdry{\Omega}} > 0.
\]

\begin{lemma} \label{Lemma 11}
There exist constants $0 < r_0 \le \delta_0$ and $\lambda > 0$ such that whenever $x_0 \in \set{u > 1}$ and $r := \dist{x_0}{\set{u \le 1}} \le r_0$, there is a point $x_1 \in \bdry{B_r(x_0)}$ satisfying
\[
u(x_1) \ge 1 + (1 + \lambda)\, (u(x_0) - 1).
\]
\end{lemma}

\begin{proof}
Suppose not. Then there exist sequences $\lambda_j \searrow 0$ and $x_j \in \set{u > 1}$ with $r_j := \dist{x_j}{\set{u \le 1}} \to 0$ such that
\[
\max_{x \in \bdry{B_{r_j}(x_j)}}\, u(x) < 1 + (1 + \lambda_j)\, (u(x_j) - 1).
\]
Since $u$ is nondegenerate, we may assume that $u(x_j) \ge 1 + cr_j$ for some constant $c > 0$. Noting that $B_{r_j}(x_j) \subset \set{u > 1}$ and $\exists\, x_j' \in \bdry{B_{r_j}(x_j)}$ such that $u(x_j') = 1$, set
\[
v_j(y) = \frac{1}{r_j}\, (u(x_j + r_j\, y) - 1), \quad y_j = \frac{1}{r_j}\, (x_j' - x_j).
\]
Then $v_j \in C(\closure{B_1(0)}) \cap C^2(B_1(0))$ satisfies
\begin{gather}
\label{38} - \Delta v_j = r_j\, g(x_j + r_j\, y,r_j\, v_j) \quad \text{in } B_1(0),\\[7.5pt]
\label{39} \max_{y \in \bdry{B_1(0)}}\, v_j(y) < (1 + \lambda_j)\, v_j(0),\\[7.5pt]
\label{40} v_j(0) \ge c, \quad v_j(y_j) = 0.
\end{gather}

We have
\begin{gather}
\label{41} 0 \le v_j(y) = \frac{1}{r_j}\, (u(x_j + r_j\, y) - u(x_j')) \le \frac{L}{r_j}\, |x_j + r_j\, y - x_j'| \le 2L \quad \forall y \in \closure{B_1(0)},\\[7.5pt]
|v_j(y) - v_j(z)| = \frac{1}{r_j}\, |u(x_j + r_j\, y) - u(x_j + r_j\, z)| \le L\, |y - z| \quad \forall y, z \in \closure{B_1(0)}, \notag\\[7.5pt]
\int_{B_1(0)} |\nabla v_j(y)|^2\, dy = r_j^{-N} \int_{B_{r_j}(x_j)} |\nabla u(x)|^2\, dx \le c_1 L^2 \notag
\end{gather}
for some $c_1 > 0$. Thus, for suitable subsequences, $v_j$ converges uniformly on $\closure{B_1(0)}$ and weakly in $H^1(B_1(0))$ to a Lipschitz continuous function $v$ and $y_j$ converges to a point $y_0 \in \bdry{B_1(0)}$. For any $\varphi \in C^\infty_0(B_1(0))$, testing \eqref{38} with $\varphi$ gives
\[
\int_{B_1(0)} \nabla v_j \cdot \nabla \varphi\, dx = r_j \int_{B_1(0)} g(x_j + r_j\, y,r_j\, v_j)\, \varphi\, dx,
\]
and passing to the limit gives
\[
\int_{B_1(0)} \nabla v \cdot \nabla \varphi\, dx = 0
\]
since $r_j \to 0$ and $g(x_j + r_j\, y,r_j\, v_j)$ is bounded by \eqref{41} and $(g_2)$. So $v$ is harmonic in $B_1(0)$. By \eqref{39},
\[
\max_{y \in \bdry{B_1(0)}}\, v(y) \le v(0),
\]
and hence $v$ is constant by the maximum principle. On the other hand,
\[
v(0) \ge c > 0 = v(y_0)
\]
by \eqref{40}, which is impossible for constant $v$.
\end{proof}

\begin{lemma} \label{Lemma 12}
There exist constants $0 < r_0 \le \delta_0$ and $\gamma > 0$ such that whenever $x_0 \in F(u)$ and $0 < r \le r_0$, there is a point $x \in B_r(x_0) \setminus B_{r/2}(x_0)$ satisfying
\[
u(x) \ge 1 + \gamma r.
\]
\end{lemma}

\begin{proof}
Let $r_0$ and $\lambda$ be as in Lemma \ref{Lemma 11}, let $x_0 \in F(u)$, and let $0 < r \le r_0$. Since $x_0 \in F(u)$, $\exists\, x_1 \in B_{r/4}(x_0)$ such that $u(x_1) > 1$. Then
\[
r_1 := \dist{x_1}{\set{u \le 1}} \le |x_1 - x_0| < \frac{r}{4} < r_0,
\]
so $\exists\, x_2 \in \bdry{B_{r_1}(x_1)}$ such that
\[
u(x_2) - 1 \ge (1 + \lambda)\, (u(x_1) - 1)
\]
by Lemma \ref{Lemma 11}. Repeating the argument, for $j = 1,\dots,k$, $\exists\, x_{j+1} \in \bdry{B_{r_j}(x_j)},\, r_j = \dist{x_j}{\set{u \le 1}}$ such that
\begin{equation} \label{43}
u(x_{j+1}) - 1 \ge (1 + \lambda)\, (u(x_j) - 1),
\end{equation}
provided that $x_2,\dots,x_k \in B_r(x_0)$. Then
\[
u(x_{k+1}) - 1 \ge (1 + \lambda)^k\, (u(x_1) - 1).
\]
Since $u$ is bounded, it follows that eventually $x_{k+1} \notin B_r(x_0)$.

Since $u$ is nondegenerate, we may assume that $u(x_j) \ge 1 + cr_j,\, j = 1,\dots,k$ for some constant $c > 0$. Then \eqref{43} gives
\[
u(x_{j+1}) - u(x_j) \ge \lambda\, (u(x_j) - 1) \ge \lambda cr_j, \quad j = 1,\dots,k - 1,
\]
so
\begin{equation} \label{42}
u(x_k) \ge u(x_1) + \lambda c \sum_{j=1}^{k-1} r_j > 1 + \lambda c \sum_{j=1}^{k-1} |x_{j+1} - x_j| \ge 1 + \lambda c\, |x_k - x_1|
\end{equation}
by the triangle inequality. If $x_k \in B_{r/2}(x_0)$, then
\[
|x_{k+1} - x_k| = r_k = \dist{x_k}{\set{u \le 1}} \le |x_k - x_0| < \frac{r}{2}
\]
and hence
\[
|x_{k+1} - x_0| \le |x_{k+1} - x_k| + |x_k - x_0| < r,
\]
contradicting $x_{k+1} \notin B_r(x_0)$, so $x_k \notin B_{r/2}(x_0)$. Since $x_1 \in B_{r/4}(x_0)$, then $|x_k - x_1| > r/4$ and hence $u(x_k) \ge 1 + \lambda cr/4$ by \eqref{42}.
\end{proof}

\begin{lemma} \label{Lemma 13}
There exist constants $0 < r_0 \le \delta_0$ and $\lambda, \nu > 0$ such that whenever $x_0 \in F(u)$ and $0 < r \le r_0$, there is a point $x_1 \in \bdry{B_r(x_0)}$ satisfying
\begin{equation} \label{100}
u(x_1) - 1 \ge \lambda r,
\end{equation}
and
\begin{equation} \label{65}
\dashint_{\bdry{B_r(x_0)}} (u - 1)_+\, d\sigma \ge \nu r.
\end{equation}
\end{lemma}

\begin{proof}
Suppose there are no such $r_0$ and $\lambda$. Then there exist sequences $r_j, \lambda_j \searrow 0$ and $x_j \in F(u)$ such that
\[
u - 1 < \lambda_j\, r_j \quad \text{on } \bdry{B_{r_j}(x_j)}.
\]
Since
\[
(u(x) - 1)_+ \le |u(x) - u(x_j)| \le L\, |x - x_j| < L r_j \quad \forall x \in B_{r_j}(x_j)
\]
and $g(x,s)$ is nondecreasing in $s$ by $(g_4)$,
\[
- \Delta (u - 1) \le g(x,(u - 1)_+) \le g(x,L r_j) \le r_j^{\mu - 1} g(x,L) \quad \text{in } B_{r_j}(x_j)
\]
for all sufficiently large $j$ by \eqref{104}. Consider the barrier function $v_j$ solving the problem
\[
\left\{\begin{aligned}
- \Delta v_j & = r_j^{\mu - 1} g(x,L) && \text{in } B_{r_j}(x_j)\\[10pt]
v_j & = \lambda_j\, r_j && \text{on } \bdry{B_{r_j}(x_j)}.
\end{aligned}\right.
\]
Since $g(x,L)$ is bounded by $(g_2)$,
\[
v_j \le \lambda_j\, r_j + c_2 r_j^{\mu - 1}
\]
for some $c_2 > 0$. On the other hand, there exist a constant $\gamma > 0$ and for all sufficiently large $j$ a point $y_j \in B_{r_j}(x_j) \setminus B_{r_j/2}(x_j)$ such that $u(y_j) - 1 \ge \gamma r_j$ by Lemma \ref{Lemma 12}. Since $u(y_j) - 1 \le v_j(y_j)$, then
\[
0 < \gamma \le \lambda_j + c_2 r_j^{\mu - 2}
\]
and the last expression goes to zero since $\mu > 2$, a contradiction.

Let $x_1 \in \bdry{B_r(x_0)}$ satisfy \eqref{100}. Then $u - 1 > \lambda r/2$ in $B_{\lambda r/2L}(x_1)$ by Lipschitz continuity, which implies \eqref{65} for some constant $\nu > 0$ depending on $N$, $\lambda$, and $L$.
\end{proof}

\begin{lemma} \label{Lemma 14}
There exist constants $0 < r_0 \le \delta_0$ and $\eps_0, c, C > 0$, depending on $N$ and $L$, such that whenever $x_0 \in F(u)$, $0 < r \le r_0$,
\begin{equation} \label{64}
\sigma(\set{u \le 1} \cap \bdry{B_r(x_0)}) \le \eps r^{N-1}
\end{equation}
for some $0 < \eps < \eps_0$, where $\sigma$ denotes the $(N - 1)$-dimensional Hausdorff measure, and $v$ is the harmonic function in $B_r(x_0)$ with $v = u$ on $\bdry{B_r(x_0)}$, we have
\begin{gather}
\label{61} \int_{B_r(x_0)} |\nabla v|^2\, dx \le \int_{B_r(x_0)} |\nabla u|^2\, dx - cr^N,\\[7.5pt]
\label{63} \int_{B_r(x_0)} |\nabla v^-|^2\, dx \le C \eps^{1/N} r^N,\\[7.5pt]
\label{62} \int_{B_r(x_0)} u^+ g(x,u^+)\, dx + \int_{B_r(x_0)} v^+ g(x,v^+)\, dx \le Cr^{\mu + N},\\[7.5pt]
\label{108} \int_{B_r(x_0)} G(x,u^+)\, dx + \int_{B_r(x_0)} G(x,v^+)\, dx \le Cr^{\mu + N}.
\end{gather}
\end{lemma}

\begin{proof}
Let $r_0$ be as in Lemma \ref{Lemma 13}. The estimates \eqref{61} and \eqref{63} follow exactly as in Jerison and Perera \cite[Lemma 7.3]{MR3790500}. Since
\[
|u(x) - 1| = |u(x) - u(x_0)| \le L\, |x - x_0| < Lr \quad \forall x \in B_r(x_0)
\]
and $g(x,s)$ is nondecreasing in $s$ by $(g_4)$,
\[
\mu\, G(x,u^+) \le u^+ g(x,u^+) \le Lr\, g(x,Lr) \le Lr^\mu g(x,L) \quad \text{in } B_r(x_0)
\]
for $r \le 1$ by \eqref{103} and \eqref{104}. Since $v$ is harmonic in $B_r(x_0)$ and equals $u$ on $\bdry{B_r(x_0)}$, where $|u - 1| \le Lr$, we also have $|v - 1| \le Lr$ in $B_r(x_0)$ by the maximum principle, so similar estimates hold for $G(x,v^+)$ and $v^+ g(x,v^+)$. So \eqref{62} and \eqref{108} follow from $(g_2)$.
\end{proof}

We are now ready to prove the second part of Proposition \ref{Proposition 3}.

\begin{proof}[Proof of the positive density property in Proposition \ref{Proposition 3}]
Let $r_0$ be as in Lemma \ref{Lemma 14}, let $x_0 \in F(u)$, and let $0 < r \le r_0$. Taking $r_0$ smaller if necessary, there exist a constant $\gamma > 0$ and a point $x_1 \in B_{r/2}(x_0) \setminus B_{r/4}(x_0)$ such that $u(x_1) \ge 1 + \gamma r/2$ by Lemma \ref{Lemma 12}. Let $\kappa = \min \set{1/2,\gamma/2L}$. Then
\[
u(x) \ge u(x_1) - L\, |x - x_1| > 1 + \left(\frac{\gamma}{2} - L \kappa\right) r \ge 1 \quad \forall x \in B_{\kappa r}(x_1),
\]
so
\[
\frac{\vol{{\set{u > 1}} \cap B_r(x_0)}}{\vol{B_r(x_0)}} \ge \kappa^N.
\]

Suppose the above volume fraction is not bounded away from $1$. Then for arbitrarily small $\lambda, \rho > 0$, $\exists\, x_0 \in F(u)$ such that
\begin{equation} \label{68}
\vol{{\set{u \le 1}} \cap B_\rho(x_0)} < \lambda \rho^N.
\end{equation}
Then
\[
\int_{\rho/2}^\rho \sigma(\set{u \le 1} \cap \partial B_r(x_0))\, dr < \lambda \rho^N,
\]
so for some $r \in [\rho/2,\rho]$,
\[
\sigma(\set{u \le 1} \cap \partial B_r(x_0)) < 2 \lambda \rho^{N-1} \le 2^N \lambda r^{N-1},
\]
i.e., the inequality \eqref{64} in Lemma \ref{Lemma 14} holds with $\eps = 2^N \lambda$. Let $v$ be as in Lemma \ref{Lemma 14}, and set $w = v$ in $B_r(x_0)$ and $w = u$ in $\Omega \setminus B_r(x_0)$. We will show that $J(\pi(w)) < J(u)$ if $r$ and $\eps$ are sufficiently small, which is a contradiction since $\pi(w) \in \M$ and $u$ minimizes $\restr{J}{\M}$.

We have
\begin{gather}
\label{66} \int_\Omega |\nabla w|^2\, dx \le \int_\Omega |\nabla u|^2\, dx - cr^N,\\[7.5pt]
\label{72} \int_\Omega |\nabla w^-|^2\, dx \le \int_\Omega |\nabla u^-|^2\, dx + C \eps^{1/N} r^N,\\[7.5pt]
\label{67} \abs{\int_\Omega w^+ g(x,w^+)\, dx - \int_\Omega u^+ g(x,u^+)\, dx} \le Cr^{\mu + N},\\[7.5pt]
\label{107} \abs{\int_\Omega G(x,w^+)\, dx - \int_\Omega G(x,u^+)\, dx} \le Cr^{\mu + N}
\end{gather}
by \eqref{61}--\eqref{108}. By \eqref{68},
\[
\int_{B_r(x_0)} |\nabla u^-|^2\, dx \le \int_{B_\rho(x_0)} |\nabla u^-|^2\, dx < L^2 \lambda \rho^N \le L^2 \eps r^N,
\]
which together with \eqref{66} gives
\begin{equation} \label{70}
\int_\Omega |\nabla w^+|^2\, dx \le \int_\Omega |\nabla u^+|^2\, dx - \frac{c}{2}\, r^N
\end{equation}
for sufficiently small $\eps$. Estimate \eqref{68} also implies
\begin{multline} \label{73}
\vol{{\set{w > 1}}} \le \vol{{\set{u > 1}} \setminus B_r(x_0)} + \vol{B_r(x_0)} = \vol{{\set{u > 1}}}\\[7.5pt]
+ \vol{{\set{u \le 1}} \cap B_r(x_0)} < \vol{{\set{u > 1}}} + \lambda \rho^N \le \vol{{\set{u > 1}}} + \eps r^N.
\end{multline}

Since $u, \pi(w) \in \M$,
\begin{equation} \label{109}
\int_\Omega |\nabla u^+|^2\, dx = \int_\Omega u^+ g(x,u^+)\, dx
\end{equation}
and
\begin{equation} \label{110}
\int_\Omega t_w^2\, |\nabla w^+|^2\, dx = \int_\Omega t_w w^+ g(x,t_w w^+)\, dx.
\end{equation}
By \eqref{70}, \eqref{109}, and \eqref{67},
\begin{multline*}
\int_\Omega |\nabla w^+|^2\, dx \le \int_\Omega u^+ g(x,u^+)\, dx - \frac{c}{2}\, r^N\\[7.5pt]
\le \int_\Omega w^+ g(x,w^+)\, dx - \left(\frac{c}{2} - Cr^\mu\right) r^N \le \int_\Omega w^+ g(x,w^+)\, dx
\end{multline*}
for sufficiently small $r$. Then $t_w \le 1$ by Lemma \ref{Lemma 2}. Combining this with \eqref{112}, \eqref{110}, and $(g_4)$ gives
\begin{eqnarray*}
J(\pi(w)) & = & \int_\Omega \bigg[\half\, |\nabla w^-|^2 + \left(\half - \frac{1}{\mu}\right) t_w^2\, |\nabla w^+|^2 + \frac{1}{\mu}\, t_w w^+ g(x,t_w w^+)\\[7.5pt]
& & - G(x,t_w w^+)\bigg]\, dx + \vol{{\set{w > 1}}}\\[7.5pt]
& \le & \int_\Omega \left[\half\, |\nabla w^-|^2 + \left(\half - \frac{1}{\mu}\right) |\nabla w^+|^2 + \frac{1}{\mu}\, w^+ g(x,w^+) - G(x,w^+)\right] dx\\[7.5pt]
& & + \vol{{\set{w > 1}}}.
\end{eqnarray*}
This together with the estimates \eqref{72}--\eqref{73} and \eqref{109} gives
\begin{eqnarray*}
J(\pi(w)) & \le & \int_\Omega \left[\half\, |\nabla u^-|^2 + \left(\half - \frac{1}{\mu}\right) |\nabla u^+|^2 + \frac{1}{\mu}\, u^+ g(x,u^+) - G(x,u^+)\right] dx\\[7.5pt]
& & + \vol{{\set{u > 1}}} + \half\, C \eps^{1/N} r^N - \left(\half - \frac{1}{\mu}\right) \frac{c}{2}\, r^N + \left(\frac{1}{\mu} + 1\right) Cr^{\mu + N} + \eps r^N\\[7.5pt]
& = & J(u) - \left[\left(\half - \frac{1}{\mu}\right) \frac{c}{2} - \half\, C \eps^{1/N} - \left(\frac{1}{\mu} + 1\right) Cr^\mu - \eps\right] r^N\\[7.5pt]
& < & J(u)
\end{eqnarray*}
for sufficiently small $r$ and $\eps$.
\end{proof}

\section{Subcritical case} \label{Subcritical}

In this section we prove Theorem \ref{Theorem 2}. Let $\eps_j \searrow 0$. We will show that each approximating functional $J_{\eps_j}$ has a critical point $u_j$ of mountain pass type and apply Theorem \ref{Theorem 1}.

Recall that $J_{\eps_j}$ satisfies the Palais-Smale compactness condition at the level $c \in \R$, or the \PS{c} condition for short, if every sequence $\seq{u_k}$ in $H^1_0(\Omega)$ such that $J_{\eps_j}(u_k) \to c$ and $J_{\eps_j}'(u_k) \to 0$ as $k \to \infty$, called a \PS{c} sequence for $J_{\eps_j}$, has a strongly convergent subsequence.

\begin{lemma} \label{Lemma 201}
The functional $J_{\eps_j}$ satisfies the {\em \PS{c}} condition for all $c \in \R$.
\end{lemma}

\begin{proof}
Let $\seq{u_k}$ be a \PS{c} sequence for $J_{\eps_j}$. Since $g$ satisfies the subcritical growth condition $(g_2')$, it suffices to show that $\seq{u_k}$ is bounded in $H^1_0(\Omega)$ by a standard argument. We have
\begin{equation} \label{200}
J_{\eps_j}(u_k) = \int_\Omega \left[\half\, |\nabla u_k|^2 + \B\left(\frac{u_k - 1}{\eps_j}\right) - G(x,(u_k - 1)_+)\right] dx = c + \o(1)
\end{equation}
and
\begin{equation} \label{201}
J_{\eps_j}'(u_k)\, v = \int_\Omega \left[\nabla u_k \cdot \nabla v + \frac{1}{\eps_j}\, \beta\left(\frac{u_k - 1}{\eps_j}\right) v - g(x,(u_k - 1)_+)\, v\right] dx = \o(\norm{v})
\end{equation}
for all $v \in H^1_0(\Omega)$. Subtracting \eqref{201} with $v = u_k^+/\mu$ from \eqref{200} and using \eqref{103} gives
\begin{multline*}
\left(\half - \frac{1}{\mu}\right) \int_\Omega |\nabla u_k^+|^2\, dx + \half \int_\Omega |\nabla u_k^-|^2\, dx + \int_\Omega \left[\B\left(\frac{u_k^+}{\eps_j}\right) - \frac{1}{\mu}\, \beta\left(\frac{u_k^+}{\eps_j}\right) \frac{u_k^+}{\eps_j}\right] dx\\[7.5pt]
\le c + \o(\|u_k^+\| + 1).
\end{multline*}
Since $\mu > 2$, and $\B(t) \ge 0$ and $\beta(t)\, t \le 2$ for all $t \in \R$, it follows from this that $\|u_k^\pm\|$, and hence also $\norm{u_k}$, is bounded.
\end{proof}

Next we show that $J_{\eps_j}$ has the mountain pass geometry. Clearly, $J_{\eps_j}(0) = 0$.

\begin{lemma} \label{Lemma 200}
There exist $\rho > 0$ and $u_1 \in H^1_0(\Omega)$ with $\norm{u_1} > \rho$ such that
\begin{equation} \label{202}
\norm{u} \le \rho \implies J_{\eps_j}(u) \ge \frac{1}{3} \norm{u}^2
\end{equation}
and
\begin{equation} \label{203}
J_{\eps_j}(u_1) < 0.
\end{equation}
\end{lemma}

\begin{proof}
By $(g_2')$,
\[
G(x,s) \le a_3\, (1 + s^p) \quad \forall (x,s) \in \Omega \times [0,\infty)
\]
for some $a_3 > 0$. Noting that $(s - 1)_+ \le |s|$ for all $s \in \R$, this gives
\[
G(x,(s - 1)_+) \le 2 a_3\, |s|^p \quad \forall (x,s) \in \Omega \times \R.
\]
Since $\B$ is nonnegative, then
\[
J_{\eps_j}(u) \ge \half \int_\Omega |\nabla u|^2\, dx - 2 a_3 \int_\Omega |u|^p\, dx \quad \forall u \in H^1_0(\Omega),
\]
and this together with the Sobolev embedding theorem gives \eqref{202} since $p > 2$.

To prove \eqref{203}, take a function $u_0 \in H^1_0(\Omega)$ with $u_0^+ \ne 0$ and let $u_1 = u_0^- + t u_0^+,\, t > 0$. Then
\[
J_{\eps_j}(u_1) = \int_\Omega \left[\half\, |\nabla u_0^-|^2 + \frac{t^2}{2}\, |\nabla u_0^+|^2 + \B\left(\frac{u_1 - 1}{\eps_j}\right) - G(x,t u_0^+)\right] dx \to - \infty
\]
as $t \to \infty$ since $\B$ is bounded and $G$ satisfies \eqref{1031} with $\mu > 2$, so \eqref{203} holds for sufficiently large $t$.
\end{proof}

By Lemma \ref{Lemma 200}, the class of paths
\[
\Gamma_j = \set{\gamma \in C([0,1],H^1_0(\Omega)) : \gamma(0) = 0,\, J_{\eps_j}(\gamma(1)) < 0}
\]
is nonempty and
\begin{equation} \label{22}
c_j := \inf_{\gamma \in \Gamma_j}\, \max_{u \in \gamma([0,1])}\, J_{\eps_j}(u) \ge \frac{1}{3}\, \rho^2.
\end{equation}

\begin{lemma} \label{Lemma 600}
The functional $J_{\eps_j}$ has a (nontrivial) critical point $u_j$ at the level $c_j$.
\end{lemma}

\begin{proof}
For $a \in \R$, set
\[
J_{\eps_j}^a = \set{u \in H^1_0(\Omega) : J_{\eps_j}(u) \le a}.
\]
Since $J_{\eps_j}$ satisfies the \PS{c_j} condition by Lemma \ref{Lemma 201}, if $J_{\eps_j}$ has no critical point at the level $c_j$, then a standard deformation argument gives a constant $0 < \delta \le c_j/2$ and a continuous map
\[
\eta : J_{\eps_j}^{c_j + \delta} \to J_{\eps_j}^{c_j - \delta}
\]
such that $\eta$ is the identity on $J_{\eps_j}^0$ (see, e.g., Perera and Schechter \cite[Lemma 1.3.3]{MR3012848}). By the definition of $c_j$, there exists a path $\gamma \in \Gamma_j$ such that
\[
\max_{u \in \gamma([0,1])}\, J_{\eps_j}(u) \le c_j + \delta.
\]
Then $\widetilde{\gamma} := \eta \circ \gamma \in \Gamma_j$ and
\[
\max_{u \in \widetilde{\gamma}([0,1])}\, J_{\eps_j}(u) \le c_j - \delta,
\]
contradicting the definition of $c_j$.
\end{proof}

Since $u_j$ is nontrivial, it is positive in $\Omega$ as noted in the introduction. Next we note that $c_j$ is below the mountain pass level of $J$ given by \eqref{301}.

\begin{lemma} \label{Lemma 1}
We have $c_j \le c$, in particular, $c > 0$.
\end{lemma}

\begin{proof}
Since $\B((t - 1)/\eps_j) \le \goodchi_{\set{t > 1}}$ for all $t \in \R$, $J_{\eps_j}(u) \le J(u)$ for all $u \in H^1_0(\Omega)$. So $\Gamma \subset \Gamma_j$ and
\[
c_j \le \max_{u \in \gamma([0,1])}\, J_{\eps_j}(u) \le \max_{u \in \gamma([0,1])}\, J(u)
\]
for all $\gamma \in \Gamma$.
\end{proof}

Now we obtain the a priori estimates needed to apply Theorem \ref{Theorem 1}.

\begin{lemma}
The sequence $\seq{u_j}$ is bounded in $H^1_0(\Omega) \cap L^\infty(\Omega)$.
\end{lemma}

\begin{proof}
As in the proof of Lemma \ref{Lemma 201},
\[
\left(\half - \frac{1}{\mu}\right) \int_\Omega |\nabla u_j^+|^2\, dx + \half \int_\Omega |\nabla u_j^-|^2\, dx + \int_\Omega \left[\B\left(\frac{u_j^+}{\eps_j}\right) - \frac{1}{\mu}\, \beta\left(\frac{u_j^+}{\eps_j}\right) \frac{u_j^+}{\eps_j}\right] dx \le c_j.
\]
Since $\mu > 2$, $\B(t) \ge 0$ and $\beta(t)\, t \le 2$ for all $t \in \R$, and $c_j \le c$ by Lemma \ref{Lemma 1}, it follows from this that $\seq{u_j}$ is bounded in $H^1_0(\Omega)$.

We have
\[
- \Delta u_j = - \frac{1}{\eps_j}\, \beta\left(\frac{u_j^+}{\eps_j}\right) + g(x,u_j^+) \le (a_1 + a_2)\, (u_j^+)^{p-1} \le (a_1 + a_2)\, u_j^{p-1}
\]
by $(g_1)$ and $(g_2')$. This together with the fact that $\seq{u_j}$ is bounded in $H^1_0(\Omega)$ implies that $\seq{u_j}$ is also bounded in $L^\infty(\Omega)$ (see Bonforte et al.\! \cite{MR3155966}).
\end{proof}

We are now ready to prove Theorem \ref{Theorem 2}.

\begin{proof}[Proof of Theorem \ref{Theorem 2}]
Let $\seq{u_j}$ be the sequence of critical points of the functionals $J_{\eps_j}$ constructed above, and let $u \in H^1_0(\Omega) \cap C^2(\closure{\Omega} \setminus F(u))$ be the Lipschitz continuous limit of a suitable subsequence of $\seq{u_j}$ given by Theorem \ref{Theorem 1}. Since $J_{\eps_j}(u_j) = c_j$,
\[
\varlimsup J_{\eps_j}(u_j) \ge \frac{1}{3}\, \rho^2 > 0
\]
by \eqref{22}, so $u$ is nontrivial by Theorem \ref{Theorem 1} \ref{Theorem 1.iii}. Since $u$ satisfies the equation $- \Delta u = g(x,(u - 1)_+)$ in $\Omega \setminus F(u)$, if $u \le 1$ everywhere, then $u$ is harmonic in $\Omega$ by $(g_1)$ and hence vanishes identically by the maximum principle, so $u > 1$ in a nonempty open subset of $\Omega$. In the interior of $\set{u \le 1}$, $u$ is harmonic with boundary values $0$ on $\bdry{\Omega}$ and $1$ on $F(u)$, so $u$ is positive in $\Omega$. In the set $\set{u > 1}$, $u$ satisfies the equation $- \Delta u = g(x,u - 1)$, and integrating this equation multiplied by $u - 1$ shows that $u \in \M$. Combining this with Proposition \ref{Proposition 4}, Theorem \ref{Theorem 1} \ref{Theorem 1.iii}, and Lemma \ref{Lemma 1} gives
\[
c \le \inf_{v \in \M}\, J(v) \le J(u) \le \varliminf J_{\eps_j}(u_j) \le \varlimsup J_{\eps_j}(u_j) \le c,
\]
so
\[
c = \inf_{v \in \M}\, J(v) = J(u) = \lim J_{\eps_j}(u_j).
\]
Then $u$ is a mountain pass point of $J$ by Proposition \ref{Proposition 4}. Moreover, since $u$ minimizes $\restr{J}{\M}$ and satisfies the inequality $- \Delta u \le g(x,(u - 1)_+)$ in the distributional sense in $\Omega$, $u$ is nondegenerate and has the positive density property for $\set{u > 1}$ and $\set{u \le 1}$ by Proposition \ref{Proposition 3}. Hence $u$ satisfies the free boundary condition in the viscosity sense by Proposition \ref{Proposition 1}. Since $J_{\eps_j}(u_j) \to J(u)$, $u$ also satisfies the free boundary condition in the variational sense by Proposition \ref{Proposition 2}. So it follows from Theorem \ref{Theorem 4} that the free boundary $F(u)$ has finite $(N - 1)$-dimensional Hausdorff measure and is a $C^\infty$-hypersurface except on a closed set of Hausdorff dimension at most $N - 3$, and that near the smooth part of $F(u)$, $(u - 1)_\pm$ are smooth and the free boundary condition is satisfied in the classical sense.
\end{proof}

\section{Critical case} \label{Critical}

In this section we prove Theorem \ref{Theorem 3}. Let $\eps_j \searrow 0$. As in the last section, we will show that each approximating functional
\[
J_{\eps_j}(u) = \int_\Omega \left[\half\, |\nabla u|^2 + \B\left(\frac{u - 1}{\eps_j}\right) - \frac{\kappa}{2^\ast}\, (u^+)^{2^\ast} - \frac{\lambda}{\mu}\, (u^+)^\mu\right] dx
\]
has a critical point $u_j$ of mountain pass type and apply Theorem \ref{Theorem 1}.

Let
\begin{equation} \label{305}
S = \inf_{u \in H^1_0(\Omega) \setminus \set{0}}\, \frac{\dint_\Omega |\nabla u|^2\, dx}{\left(\dint_\Omega |u|^{2^\ast}\, dx\right)^{2/2^\ast}}
\end{equation}
be the best constant for the Sobolev embedding $H^1_0(\Omega) \hookrightarrow L^{2^\ast}(\Omega)$. The functional $J_{\eps_j}$ has the following compactness property.

\begin{lemma} \label{Lemma 301}
If
\[
0 < c < \frac{1}{N}\, \frac{S^{N/2}}{\kappa^{N/2 - 1}},
\]
then every {\em \PS{c}} sequence for $J_{\eps_j}$ has a subsequence that converges weakly to a nontrivial critical point $u$ of $J_{\eps_j}$ satisfying $J_{\eps_j}(u) \le c$.
\end{lemma}

\begin{proof}
Let $0 < c < S^{N/2}/N \kappa^{N/2 - 1}$ and let $\seq{u_k}$ be a \PS{c} sequence for $J_{\eps_j}$. We have
\begin{equation} \label{300}
J_{\eps_j}(u_k) = \int_\Omega \left[\half\, |\nabla u_k|^2 + \B\left(\frac{u_k - 1}{\eps_j}\right) - \frac{\kappa}{2^\ast}\, (u_k^+)^{2^\ast} - \frac{\lambda}{\mu}\, (u_k^+)^\mu\right] dx = c + \o(1)
\end{equation}
and
\begin{multline} \label{302}
J_{\eps_j}'(u_k)\, v = \int_\Omega \left[\nabla u_k \cdot \nabla v + \frac{1}{\eps_j}\, \beta\left(\frac{u_k - 1}{\eps_j}\right) v - \kappa\, (u_k^+)^{2^\ast - 1}\, v - \lambda\, (u_k^+)^{\mu - 1}\, v\right] dx\\[7.5pt]
= \o(\norm{v})
\end{multline}
for all $v \in H^1_0(\Omega)$. Subtracting \eqref{302} with $v = u_k^+/\mu$ from \eqref{300} gives
\begin{multline*}
\left(\half - \frac{1}{\mu}\right) \int_\Omega |\nabla u_k^+|^2\, dx + \half \int_\Omega |\nabla u_k^-|^2\, dx + \int_\Omega \left[\B\left(\frac{u_k^+}{\eps_j}\right) - \frac{1}{\mu}\, \beta\left(\frac{u_k^+}{\eps_j}\right) \frac{u_k^+}{\eps_j}\right] dx\\[7.5pt]
+ \kappa \left(\frac{1}{\mu} - \frac{1}{2^\ast}\right) \int_\Omega (u_k^+)^{2^\ast}\, dx = c + \o(\|u_k^+\| + 1).
\end{multline*}
Since $2 < \mu < 2^\ast$, and $\B(t) \ge 0$ and $\beta(t)\, t \le 2$ for all $t \in \R$, it follows from this that $\|u_k^\pm\|$, and hence also $\norm{u_k}$, is bounded. So a renamed subsequence of $\seq{u_k}$ converges to some $u$ weakly in $H^1_0(\Omega)$, strongly in $L^q(\Omega)$ for all $1 \le q < 2^\ast$, and a.e.\! in $\Omega$. Since $\beta$ is bounded, passing to the limit in \eqref{302} gives
\begin{equation} \label{303}
\int_\Omega \left[\nabla u \cdot \nabla v + \frac{1}{\eps_j}\, \beta\left(\frac{u - 1}{\eps_j}\right) v - \kappa\, (u^+)^{2^\ast - 1}\, v - \lambda\, (u^+)^{\mu - 1}\, v\right] dx = 0
\end{equation}
for all $v \in C^\infty_0(\Omega)$, and hence also for all $v \in H^1_0(\Omega)$ by density. So $u$ is a critical point of $J_{\eps_j}$.

Suppose $u = 0$. Then \eqref{300} and \eqref{302} with $v = u_k$ reduce to
\[
\frac{1}{2} \int_\Omega |\nabla u_k|^2\, dx - \frac{\kappa}{2^\ast} \int_\Omega (u_k^+)^{2^\ast}\, dx = c + \o(1), \qquad \int_\Omega |\nabla u_k|^2\, dx - \kappa \int_\Omega (u_k^+)^{2^\ast}\, dx = \o(1),
\]
respectively, so
\[
\int_\Omega |\nabla u_k|^2\, dx = Nc + \o(1), \qquad \int_\Omega (u_k^+)^{2^\ast}\, dx = \frac{Nc}{\kappa} + \o(1).
\]
Combining this with
\[
\int_\Omega |\nabla u_k|^2\, dx \ge S \left(\int_\Omega |u_k|^{2^\ast}\, dx\right)^{2/2^\ast} \ge S \left(\int_\Omega (u_k^+)^{2^\ast}\, dx\right)^{2/2^\ast}
\]
gives $c \ge S^{N/2}/N \kappa^{N/2 - 1}$, a contradiction. So $u$ is nontrivial.

It remains to show that $J_{\eps_j}(u) \le c$. We have
\[
J_{\eps_j}(u) = J_{\eps_j}(u) - \frac{1}{2^\ast}\, J_{\eps_j}'(u)\, u = \frac{1}{N} \int_\Omega |\nabla u|^2\, dx + K(u)
\]
and
\[
c = J_{\eps_j}(u_k) - \frac{1}{2^\ast}\, J_{\eps_j}'(u_k)\, u_k + \o(1) = \frac{1}{N} \int_\Omega |\nabla u_k|^2\, dx + K(u_k) + \o(1),
\]
where
\begin{multline*}
K(u) = \int_\Omega \bigg[\B\left(\frac{u - 1}{\eps_j}\right) - \frac{1}{2^\ast\, \eps_j}\, \beta\left(\frac{u - 1}{\eps_j}\right) u + \frac{\kappa}{2^\ast}\, (u^+)^{2^\ast - 1} - \frac{\lambda}{\mu}\, (u^+)^\mu\\[7.5pt]
+ \frac{\lambda}{2^\ast}\, (u^+)^{\mu - 1}\, u\bigg]\, dx.
\end{multline*}
Since $u_k \wto u$ and $K(u_k) \to K(u)$, the desired conclusion follows.
\end{proof}

As in the last section, $J_{\eps_j}$ has the mountain pass geometry and the minimax level $c_j$ defined in \eqref{22} satisfies
\begin{equation} \label{601}
\inf c_j > 0.
\end{equation}
Let $u_0 \in H^1_0(\Omega)$ with $u_0^+ \ne 0$. For $t > 0$,
\[
J_{\eps_j}(u_0^- + t u_0^+) \le \int_\Omega \left[\half\, |\nabla u_0^-|^2 + \frac{t^2}{2}\, |\nabla u_0^+|^2 - \frac{\lambda t^\mu}{\mu}\, (u_0^+)^\mu\right] dx + \vol{\Omega}
\]
since $\B \le 1$ and $\kappa > 0$. Since $\mu > 2$, the right-hand side goes to $- \infty$ as $t \to \infty$, so there exists a $t_1 > 0$ such that $J_{\eps_j}(u_0^- + t_1 u_0^+) < 0$ for all $j$. Let $\gamma$ be any path joining the origin to $u_0^- + t_1 u_0^+$. Then $\gamma \in \Gamma_j$ and hence
\[
c_j \le \max_{u \in \gamma([0,1])}\, J_{\eps_j}(u) \le \max_{u \in \gamma([0,1])} \int_\Omega \left[\half\, |\nabla u|^2 - \frac{\lambda}{\mu}\, (u^+)^\mu\right] dx + \vol{\Omega}.
\]
Since the last expression is independent of $j$ and $\kappa$, there exists a $\kappa^\ast > 0$ such that for $0 < \kappa < \kappa^\ast$,
\begin{equation} \label{602}
c_j < \frac{1}{N}\, \frac{S^{N/2}}{\kappa^{N/2 - 1}}
\end{equation}
for all $j$.

\begin{lemma} \label{Lemma 5.2}
For $0 < \kappa < \kappa^\ast$, $J_{\eps_j}$ has a nontrivial critical point $u_j$ satisfying $J_{\eps_j}(u_j) \le c_j$.
\end{lemma}

\begin{proof}
In view of Lemma \ref{Lemma 301}, \eqref{601}, and \eqref{602}, it suffices to show that $J_{\eps_j}$ has a \PS{c_j} sequence. Suppose not. Then $J_{\eps_j}$ satisfies the \PS{c_j} condition vacuously and hence has a critical point $u_0$ at the level $c_j$ as in the proof of Lemma \ref{Lemma 600}. But then the constant sequence $\seq{u_0}$ is a \PS{c_j} sequence for $J_{\eps_j}$, a contradiction.
\end{proof}

As in the last section, $u_j > 0$ in $\Omega$ and $c_j \le c$, in particular, $c > 0$. In order to apply Theorem \ref{Theorem 1}, first we show that $\seq{u_j}$ is bounded in $H^1_0(\Omega)$.

\begin{lemma} \label{Lemma 650}
There exists a constant $C_1 > 0$, independent of $j$ and $\kappa$, such that
\[
\norm{u_j} \le C_1 \quad \forall j.
\]
\end{lemma}

\begin{proof}
As in the proof of Lemma \ref{Lemma 301},
\begin{multline*}
\left(\half - \frac{1}{\mu}\right) \int_\Omega |\nabla u_j^+|^2\, dx + \half \int_\Omega |\nabla u_j^-|^2\, dx + \int_\Omega \left[\B\left(\frac{u_j^+}{\eps_j}\right) - \frac{1}{\mu}\, \beta\left(\frac{u_j^+}{\eps_j}\right) \frac{u_j^+}{\eps_j}\right] dx\\[7.5pt]
+ \kappa \left(\frac{1}{\mu} - \frac{1}{2^\ast}\right) \int_\Omega (u_j^+)^{2^\ast}\, dx = c_j.
\end{multline*}
Since $2 < \mu < 2^\ast$, $\B(t) \ge 0$ and $\beta(t)\, t \le 2$ for all $t \in \R$, and $c_j \le c$, it follows from this that $\norm{u_j}$ is bounded independently of $j$ and $\kappa$.
\end{proof}

Next we show that $\seq{u_j}$ is also bounded in $L^\infty(\Omega)$ if $\kappa$ is further restricted. We will make use of the following $L^\infty$ bound obtained in Perera and Silva \cite{MR2282829}.

\begin{proposition}[{\cite[Lemma A.1 \& Remark A.3]{MR2282829}}] \label{Proposition 4.3}
Let $u \in H^1_0(\Omega)$ be a weak solution of the problem
\begin{equation} \label{4.6}
\left\{\begin{aligned}
- \Delta u & = h(x,u) && \text{in } \Omega\\[10pt]
u & > 0 && \text{in } \Omega\\[10pt]
u & = 0 && \text{on } \bdry{\Omega},
\end{aligned}\right.
\end{equation}
where $h$ is a Carath\'eodory function on $\Omega \times (0,\infty)$. Assume that there exist $r > N/2$ and $a \in L^r(\Omega)$ such that $h(x,t) \le a(x)\, t$ for a.a.\! $x \in \Omega$ and all $t > 0$. Then there exists a constant $C > 0$, depending only on $\Omega$, $\pnorm[r]{a}$, and $\pnorm[2]{u}$, such that
\[
\pnorm[\infty]{u} \le C,
\]
where $\pnorm[p]{\cdot}$ denotes the $L^p$-norm.
\end{proposition}

\begin{lemma} \label{Lemma 651}
There exist constants $0 < \kappa_\ast \le \kappa^\ast$ and $C_2 > 0$, independent of $j$ and $\kappa$, such that for $0 < \kappa < \kappa_\ast$,
\[
\pnorm[\infty]{u_j} \le C_2 \quad \forall j.
\]
\end{lemma}

\begin{proof}
The function $u_j \in H^1_0(\Omega)$ is a weak solution of problem \eqref{4.6} with
\[
h(x,t) = - \frac{1}{\eps_j}\, \beta\left(\frac{t - 1}{\eps_j}\right) + \kappa\, (u_j(x) - 1)_+^{4/(N-2)}\, (t - 1)_+ + \lambda\, (u_j(x) - 1)_+^{\mu - 2}\, (t - 1)_+.
\]
We have
\[
h(x,t) \le \left[\kappa u_j(x)^{4/(N-2)} + \lambda u_j(x)^{\mu - 2}\right] t \quad \forall (x,t) \in \Omega \times (0,\infty).
\]
Let $r = N^2/2\, (N - 2)$ and $w_j = u_j^{N/(N-2)}$. Then $r > N/2$ and
\[
|u_j^{4/(N-2)}|_r = |w_j^{4/N}|_r = \pnorm[2^\ast]{w_j}^{4/N},
\]
so it suffices to show that $\pnorm[2^\ast]{w_j}$ is bounded for sufficiently small $\kappa$ by Proposition \ref{Proposition 4.3} and Lemma \ref{Lemma 650}.

By \eqref{305},
\begin{equation} \label{4.7}
S \left(\int_\Omega w_j^{2^\ast}\, dx\right)^{2/2^\ast} \le \int_\Omega |\nabla w_j|^2\, dx = \left(\frac{N}{N - 2}\right)^2 \int_\Omega u_j^{4/(N-2)}\, |\nabla u_j|^2\, dx.
\end{equation}
Since $u_j \in C^1(\closure{\Omega})$ by standard regularity arguments, we may test
\[
- \Delta u_j = - \frac{1}{\eps_j}\, \beta\left(\frac{u_j^+}{\eps_j}\right) + \kappa\, (u_j^+)^{2^\ast - 1} + \lambda\, (u_j^+)^{\mu - 1} \le (\kappa + \tau \eps^{1/\tau} \lambda)\, u_j^{2^\ast - 1} + \frac{(1 - \tau)\, \lambda}{\eps^{1/(1 - \tau)}}\, u_j,
\]
where $\tau = (\mu - 2)/(2^\ast - 2)$ and $\eps > 0$, with $u_j^{2^\ast - 1}$ to get
\begin{equation} \label{4.8}
\frac{N + 2}{N - 2} \int_\Omega u_j^{4/(N-2)}\, |\nabla u_j|^2\, dx \le (\kappa + \eps^{1/\tau} \lambda) \int_\Omega u_j^{4/(N-2)}\, w_j^2\, dx + \frac{\lambda}{\eps^{1/(1 - \tau)}} \int_\Omega u_j^{2^\ast}\, dx.
\end{equation}
Since
\[
\int_\Omega u_j^{4/(N-2)}\, w_j^2\, dx \le \left(\int_\Omega u_j^{2^\ast}\, dx\right)^{2/N} \left(\int_\Omega w_j^{2^\ast}\, dx\right)^{2/2^\ast}
\]
by the H\"{o}lder inequality, and $\pnorm[2^\ast]{u_j}$ is bounded by Lemma \ref{Lemma 650} and the Sobolev embedding, it follows from \eqref{4.7} and \eqref{4.8} that $\pnorm[2^\ast]{w_j}$ is bounded if $\kappa$ and $\eps$ are sufficiently small.
\end{proof}

We are now ready to prove Theorem \ref{Theorem 3}.

\begin{proof}[Proof of Theorem \ref{Theorem 3}]
Let $0 < \kappa < \kappa_\ast$ and let $\seq{u_j}$ be the sequence of critical points of the functionals $J_{\eps_j}$ given by Lemma \ref{Lemma 5.2}. Then $\seq{u_j}$ is bounded in $H^1_0(\Omega) \cap L^\infty(\Omega)$ by Lemmas \ref{Lemma 650} and \ref{Lemma 651}. Let $u \in H^1_0(\Omega) \cap C^2(\closure{\Omega} \setminus F(u))$ be the Lipschitz continuous limit of a suitable subsequence of $\seq{u_j}$ given by Theorem \ref{Theorem 1}. Since $J_{\eps_j}'(u_j) = 0$,
\[
J_{\eps_j}'(u_j)\, u_j = \int_\Omega \left[|\nabla u_j|^2 + \frac{1}{\eps_j}\, \beta\left(\frac{u_j^+}{\eps_j}\right) u_j - \kappa\, (u_j^+)^{2^\ast - 1}\, u_j - \lambda\, (u_j^+)^{\mu - 1}\, u_j\right] dx = 0,
\]
and since $u_j > 0$, this together with the Sobolev embedding theorem gives
\[
\norm{u_j}^2 \le \int_\Omega \left(\kappa u_j^{2^\ast} + \lambda u_j^\mu\right) dx \le C \left(\norm{u_j}^{2^\ast} + \norm{u_j}^\mu\right)
\]
for some constant $C > 0$. Since $u_j$ is nontrivial and $2 < \mu < 2^\ast$, this implies that $\inf \norm{u_j} > 0$, and since $u_j \to u$ in $H^1_0(\Omega)$, then $u$ is nontrivial. Since $u$ satisfies the equation $- \Delta u = \kappa\, (u - 1)_+^{2^\ast - 1} + \lambda\, (u - 1)_+^{\mu - 1}$ in $\Omega \setminus F(u)$, if $u \le 1$ everywhere, then $u$ is harmonic in $\Omega$ and hence vanishes identically by the maximum principle, so $u > 1$ in a nonempty open subset of $\Omega$. In the interior of $\set{u \le 1}$, $u$ is harmonic with boundary values $0$ on $\bdry{\Omega}$ and $1$ on $F(u)$, so $u$ is positive in $\Omega$. In the set $\set{u > 1}$, $u$ satisfies the equation $- \Delta u = \kappa\, (u - 1)^{2^\ast - 1} + \lambda\, (u - 1)^{\mu - 1}$, and integrating this equation multiplied by $u - 1$ shows that $u$ is in
\[
\M = \set{u \in \W : \int_\Omega |\nabla u^+|^2\, dx = \int_\Omega \left[\kappa\, (u^+)^{2^\ast} + \lambda\, (u^+)^\mu\right] dx},
\]
where $\W = \bigset{u \in H^1_0(\Omega) : u^+ \ne 0}$. Combining this with Proposition \ref{Proposition 4}, Theorem \ref{Theorem 1} \ref{Theorem 1.iii}, and
\[
J_{\eps_j}(u_j) \le c_j \le c
\]
gives
\[
c \le \inf_{v \in \M}\, J(v) \le J(u) \le \varliminf J_{\eps_j}(u_j) \le \varlimsup J_{\eps_j}(u_j) \le c,
\]
so
\[
c = \inf_{v \in \M}\, J(v) = J(u) = \lim J_{\eps_j}(u_j).
\]
Then $u$ is a mountain pass point of $J$ by Proposition \ref{Proposition 4}. Moreover, since $u$ minimizes $\restr{J}{\M}$ and satisfies the inequality $- \Delta u \le \kappa\, (u - 1)_+^{2^\ast - 1} + \lambda\, (u - 1)_+^{\mu - 1}$ in the distributional sense in $\Omega$, $u$ is nondegenerate and has the positive density property for $\set{u > 1}$ and $\set{u \le 1}$ by Proposition \ref{Proposition 3}. Hence $u$ satisfies the free boundary condition in the viscosity sense by Proposition \ref{Proposition 1}. Since $J_{\eps_j}(u_j) \to J(u)$, $u$ also satisfies the free boundary condition in the variational sense by Proposition \ref{Proposition 2}. So it follows from Theorem \ref{Theorem 4} that the free boundary $F(u)$ has finite $(N - 1)$-dimensional Hausdorff measure and is a $C^\infty$-hypersurface except on a closed set of Hausdorff dimension at most $N - 3$, and that near the smooth part of $F(u)$, $(u - 1)_\pm$ are smooth and the free boundary condition is satisfied in the classical sense.
\end{proof}

\def\cdprime{$''$}

\end{document}